\documentclass[12pt]{article}
\usepackage{amssymb}
\usepackage{mathrsfs}
\usepackage[hypertex]{hyperref}
\usepackage{amsfonts}
\usepackage{graphicx}
\usepackage{latexsym,amsmath,amssymb,amsfonts,amsthm}

\newcommand{\bcen}{\begin{center}}
\newcommand{\ecen}{\end{center}}
\newtheorem{theorem}{Theorem}[section]
\newtheorem{lemma}[theorem]{Lemma}

\newtheorem{remark}[theorem]{Remark}

\begin{document}
\setcounter{page}{1}
\title{Several isoperimetric inequalities of Dirichlet and Neumann eigenvalues of the Witten-Laplacian}
\author{Ruifeng Chen,~~ Jing Mao$^{\ast}$}

\date{}
\protect \footnotetext{\!\!\!\!\!\!\!\!\!\!\!\!{~$\ast$ Corresponding author } }
\maketitle ~~~\\[-15mm]

\begin{center}
{\footnotesize Faculty of Mathematics and Statistics,\\
 Key Laboratory of Applied
Mathematics of Hubei Province, \\
Hubei University, Wuhan 430062, China\\
 Key Laboratory of Intelligent
Sensing System and Security (Hubei
University), Ministry of Education\\
Emails: gchenruifeng@163.com (R. F. Chen), jiner120@163.com (J. Mao)
 }
\end{center}

%\\[5mm]

\begin{abstract}
In this paper, by mainly using the rearrangement technique and
suitably constructing trial functions, under the constraint of fixed
weighted volume, we can successfully obtain several isoperimetric
inequalities for the first and the second Dirichlet eigenvalues, the
first nonzero Neumann eigenvalue of the Witten-Laplacian on bounded
domains in space forms. These spectral isoperimetric inequalities
extend those classical ones (i.e. the Faber-Krahn inequality, the
Hong-Krahn-Szeg\H{o} inequality and the Szeg\H{o}-Weinberger
inequality) of the Laplacian.
 \end{abstract}

%\markright{\sl\hfill R. F. Chen,~~ J. Mao \hfill}

\section{Introduction}  \label{s1}
\renewcommand{\thesection}{\arabic{section}}
\renewcommand{\theequation}{\thesection.\arabic{equation}}
\setcounter{equation}{0}

The study of extremum problems of prescribed functionals is of great
significance in Mathematics. For instance, a famous isoperimetric
problem, which might be known for nearly all the mathematicians, in
the $n$-dimensional ($n\geq2$) Euclidean space $\mathbb{R}^n$ is to
study the following extremum problem:
 \begin{eqnarray} \label{1-1}
 \min\left\{|\partial\Omega|_{n-1}\Big{|}|\Omega|_{n}=const.\right\}
\end{eqnarray}
for bounded domains $\Omega\subset\mathbb{R}^n$ with smooth boundary
$\partial\Omega$, where, by the abuse of notations, $|\cdot|$ stands
for the Hausdorff measure of a given geometric object, and
meanwhile, if necessary, we would put the information of dimension
as the subscript of the notation $|\cdot|$ as well. The above
extremum problem can be asked in another way as follows:
\begin{itemize}
\item \emph{Among all bounded domains in $\mathbb{R}^n$ with fixed volume, which one minimizes the area functional of the boundary?}
\end{itemize}
This classical problem has been answered completely and one knows
that the unique minimizer of the area functional should be a ball
with the volume equal to $|\Omega|_{n}=const.$ -- see, e.g.,
\cite[Chapter 1]{RS} for an interesting derivation of classical
isoperimetric inequalities in Euclidean space by using the Schwarz
symmetrization. In fact, for any bounded domain $\Omega$ in
$\mathbb{R}^n$ with smooth boundary, one has:
 \begin{eqnarray} \label{1-2}
 \frac{|\partial\Omega|^{n}}{|\Omega|^{n-1}}\geq\frac{|\mathbb{S}^{n-1}|^n}{|\mathbb{B}^{n}|^{n-1}},
 \end{eqnarray}
with equality holding if and only if $\Omega$ is a Euclidean ball.
Obviously, the RHS of (\ref{1-2}) is independent of the choice of
radius for the Euclidean $n$-ball $\mathbb{B}^n$ and the
corresponding Euclidean $(n-1)$-sphere $\mathbb{S}^{n-1}$. That is
to say, the quantity $|\mathbb{S}^{n-1}|^{n}/|\mathbb{B}^{n}|^{n-1}$
is scale invariant.
 So, for
convenience and simplification, we denote by $\mathbb{B}^n$,
$\mathbb{S}^{n-1}$ the unit Euclidean $n$-ball and the unit
Euclidean $(n-1)$-sphere, respectively. By (\ref{1-2}), one easily
knows that:
\begin{itemize}
\item Among all bounded domains in $\mathbb{R}^n$ having the same
volume, Euclidean balls minimize the boundary area.

\item Among all bounded domains in $\mathbb{R}^n$ having the same
boundary area, Euclidean balls maximize the volume.
\end{itemize}
Clearly, (\ref{1-2}) gives the answer to the problem (\ref{1-1})
completely -- for a ball $B_{\Omega}$ with
$|B_{\Omega}|_{n}=|\Omega|_{n}=const.$, it follows
that\footnote{~Clearly, $\partial B_{\Omega}$ stands for the
boundary sphere of the ball $B_{\Omega}$.}
\begin{eqnarray} \label{1-3}
 |\partial\Omega|_{n-1}\geq|\partial B_{\Omega}|_{n-1},
\end{eqnarray}
with equality holding if and only if $\Omega$ is a ball in
$\mathbb{R}^n$ (which is congruent with $B_{\Omega}$). Following the
convention in \cite{IC}, we wish to call (\ref{1-2})-(\ref{1-3}) the
\emph{geometric isoperimetric inequalities}.

 The purpose of this
paper is to investigate isoperimetric inequalities from the
viewpoint of spectral quantities of the Witten-Laplacian. However,
in order to state our conclusions clearly, we wish to first recall
several classical results on the Laplacian.

Let $(M^{n},\langle\cdot,\cdot\rangle)$ be an $n$-dimensional
($n\geq2$) complete Riemannian manifold with the metric
$g:=\langle\cdot,\cdot\rangle$. Let $\Omega\subset M^{n}$ be a
bounded domain in $M^n$ with smooth\footnote{~The smoothness
assumption for the regularity of the boundary $\partial\Omega$ is
strong enough to consider the eigenvalue problems (\ref{eigen-D1})
and (\ref{eigen-N1}). For instance, a weaker regularity assumption
that $\partial\Omega$ is Lipschitz continuous can also assure the
validity about the description of discrete spectrum of the Neumann
eigenvalue problem (\ref{eigen-N1}) of the Laplacian on the fourth
page of this paper. However, the Lipschitz continuous assumption
might not be enough to consider some other geometric problems
involved Neumann eigenvalues of (\ref{eigen-N1}). Therefore, in
order to avoid putting too much attention on discussion for the
regularity of the boundary $\partial\Omega$ (which is also not
important for the topic investigated in our paper here), without
specification, we wish to assume that $\partial\Omega$ is smooth in
this paper. This setting leads to the situation that some
conclusions of this paper may still hold under a weaker regularity
assumption for the boundary $\partial\Omega$, readers who are
interested in this situation could try to seek the weakest
regularity. } boundary $\partial\Omega$. Denote by $\Delta$,
$\nabla$ the Laplace and the gradient operators on $M^n$ associated
with the metric $g$, respectively. On $\Omega$, one can consider the
Dirichlet eigenvalue problem of the Laplacian as follows
\begin{eqnarray} \label{eigen-D1}
\left\{
\begin{array}{ll}
\Delta u+\lambda u=0\qquad & \mathrm{in}~\Omega\subset M^{n}, \\[0.5mm]
u=0 \qquad & \mathrm{on}~{\partial\Omega},
\end{array}
\right.
\end{eqnarray}
which is also known as the \emph{fixed membrane problem} of the
Laplacian. In fact, for the eigenvalue problem (\ref{eigen-D1}),
when $M^{n}$ is chosen to be $\mathbb{R}^3$, this system can be used
to describe the vibration of a membrane with boundary fixed, and
this is the reason why it is called fixed membrane problem. Because
of this physical background, eigenvalues of a prescribed eigenvalue
problem of some self-adjoint differentiable elliptic operator are
called \emph{frequencies}. It is well-known that the operator
$-\Delta$ in (\ref{eigen-D1}) only has a discrete spectrum and all
the elements (i.e., eigenvalues) can be listed non-decreasingly as
follows
 \begin{eqnarray} \label{multi-1}
0<\lambda_{1}(\Omega)<\lambda_{2}(\Omega)\leq\lambda_{3}(\Omega)\leq\cdots\uparrow\infty.
 \end{eqnarray}
For each eigenvalue $\lambda_{i}(\Omega)$, $i=1,2,\cdots$, all the
nontrivial functions satisfying (\ref{eigen-D1}) form a vector
space, which  has finite dimension and is called eigenspace of
$\lambda_{i}(\Omega)$. Moreover, all the elements in this eigenspace
are called eigenfunctions belonging to $\lambda_{i}(\Omega)$. The
dimension of this eigenspace is called multiplicity of the
eigenvalue $\lambda_{i}(\Omega)$. Each eigenvalue
$\lambda_{i}(\Omega)$ in the sequence (\ref{multi-1}) is repeated
according to its multiplicity. By variational principle, the $k$-th
Dirichlet eigenvalue $\lambda_{k}(\Omega)$ is characterized as
follows
 \begin{eqnarray*}
\lambda_{k}(\Omega)=\inf\left\{\frac{\int_{\Omega}|\nabla
f|^{2}dv}{\int_{\Omega}f^{2}dv}\Bigg{|}f\in
W^{1,2}_{0}(\Omega),f\neq0,\int_{\Omega}ff_{i}dv=0\right\},
 \end{eqnarray*}
where $dv$ denotes the Riemannian volume element of $M^n$, and
$f_{i}$, $i=1,2,\cdots,k-1$, denotes an eigenfunction of
$\lambda_{i}(\Omega)$. Here, as usual, $W^{1,2}_{0}(\Omega)$ stands
for a Sobolev space, which is the completion of the set of smooth
functions (with compact support) $C^{\infty}_{0}(\Omega)$ under the
following Sobolev norm
 \begin{eqnarray} \label{Sobo}
\|f\|_{1,2}:=\left(\int_{\Omega}f^{2}dv+\int_{\Omega}|\nabla
f|^{2}dv\right)^{1/2}.
 \end{eqnarray}
See, e.g., \cite{IC} for the above fundamental facts of the
eigenvalue problem (\ref{eigen-D1}). Besides, for convenience and
without confusion, in the sequel, except specification we will write
$\lambda_{i}(\Omega)$ as $\lambda_{i}$ directly. This convention
would be also used when we meet with other possible eigenvalue
problems.

Similar to (\ref{1-1}), for bounded domains
$\Omega\subset\mathbb{R}^n$ with smooth boundary $\partial\Omega$,
$n\geq2$, it should be interesting and important to ask the
following extremum problem:
 \begin{eqnarray} \label{extre-1}
 \min\left\{\lambda_{k}(\Omega)\Big{|}|\Omega|_{n}=const.\right\}
\end{eqnarray}
for each $k=1,2,3,\cdots$. In fact, (\ref{extre-1}) is a natural and
classical isoperimetric problem in the study of Spectral Geometry.
To the best of our knowledge, for $k=1,2$, there exist affirmative
answers to the problem (\ref{extre-1}) as follows:
\begin{itemize}
\item  (Faber-Krahn inequality, \cite{GF,EK1})
$\lambda_{1}(\Omega)\geq\lambda_{1}(B_{\Omega})$, and the equality
holds if and only if $\Omega$ is a ball in $\mathbb{R}^n$ (which is
congruent with $B_{\Omega}$,
$|B_{\Omega}|_{n}=|\Omega|_{n}=const.$). That is to say, among all
bounded domains in $\mathbb{R}^n$ having the same volume, Euclidean
balls minimize the first Dirichlet eigenvalue of the Laplacian.

\item (Hong-Krahn-Szeg\H{o} inequality, \cite{IH,EK2})
$\lambda_{2}(\Omega)\geq\lambda_{1}(\widetilde{B}_{\Omega})$, where
$\widetilde{B}_{\Omega}$ is a ball in $\mathbb{R}^n$ such that
$2|\widetilde{B}_{\Omega}|_{n}=const.=|\Omega|_{n}$. That is to say,
the minimum of the second Dirichlet eigenvalue of the Laplacian on
bounded domains $\Omega$ (whose volume equals some prescribed
positive constant) should be equal to the first Dirichlet eigenvalue
of the Laplacian on a ball $\widetilde{B}_{\Omega}$ with
$|\widetilde{B}_{\Omega}|_{n}=|\Omega|_{n}/2$.
\end{itemize}
Hong-Krahn-Szeg\H{o} inequality implies that under the constraint
that the volume of bounded domains is fixed, the second Dirichlet
eigenvalue (of the Laplacian) is minimized by two balls of the same
volume. However, if one additionally requires that $\Omega$ is
connected, then under the constraint of volume fixed
($|\Omega|_{n}=const.$), this minimizer of $\lambda_{2}(\Omega)$
cannot be attained but can be approximated by the domain
$\Omega_{\epsilon}$, obtained by joining the union of the two
congruent balls (whose volumes equal $|\Omega|_{n}/2$) by a thin
pipe of width $\epsilon$ (sufficiently small) -- see \cite{AH1} for
the precise description of this interesting example and see, e.g.,
\cite{BZ,BDM} for the strict proof of this approximation (as
$\epsilon\rightarrow0$).  In two dimensional case, it has long been
conjectured that the ball minimizes $\lambda_{3}(\Omega)$, but there
did not have much progress in this direction. For higher order
Dirichlet eigenvalues, not much is known. However, there is an
interesting result we wish to mention, that is, Berger \cite{AB}
proved that for planar bounded domain $\Omega\subset\mathbb{R}^2$,
the $i$-th ($i>4$) Dirichlet eigenvalue $\lambda_{i}(\Omega)$ is not
minimized by any union of disks.

For a bounded domain $\Omega$ (with smooth boundary) on a given
complete Riemannian $n$-manifold $M^n$, one can also consider the
Neumann eigenvalue problem of the Laplacian as follows
\begin{eqnarray} \label{eigen-N1}
\left\{
\begin{array}{ll}
\Delta u+\mu u=0\qquad & \mathrm{in}~\Omega\subset M^{n}, \\[0.5mm]
\frac{\partial u}{\partial\vec{\nu}}=0 \qquad &
\mathrm{on}~{\partial\Omega},
\end{array}
\right.
\end{eqnarray}
which is also known as the \emph{free membrane problem} of the
Laplacian. Here, $\vec{\nu}$ stands for the outward unit normal
vector of the boundary $\partial\Omega$. In fact, for the eigenvalue
problem (\ref{eigen-N1}), when $M^{n}$ is chosen to be
$\mathbb{R}^3$, this system can be used to describe the vibration of
a membrane with free boundary, and this is the reason why it is
called free membrane problem. It is well-known that the operator
$-\Delta$ in (\ref{eigen-N1}) only has a discrete spectrum and all
the eigenvalues can be listed non-decreasingly as follows
 \begin{eqnarray} \label{multi-2}
0=\mu_{0}(\Omega)<\mu_{1}(\Omega)\leq\mu_{2}(\Omega)\leq\cdots\uparrow\infty.
 \end{eqnarray}
The eigenvalue $\mu_{0}(\Omega)=0$ has nonzero constant functions as
its eigenfunctions. Each eigenvalue $\mu_{i}(\Omega)$ in the
sequence (\ref{multi-2}) is repeated according to its multiplicity
(which is finite and actually equals the dimension of
$\mu_{i}(\Omega)$'s eigenspace). By variational principle, the
$k$-th nonzero Neumann eigenvalue $\mu_{k}(\Omega)$ is characterized
as follows
\begin{eqnarray*}
\mu_{k}(\Omega)=\inf\left\{\frac{\int_{\Omega}|\nabla
f|^{2}dv}{\int_{\Omega}f^{2}dv}\Bigg{|}f\in
W^{1,2}(\Omega),f\neq0,\int_{\Omega}ff_{i}dv=0\right\},
 \end{eqnarray*}
where $f_{i}$, $i=0,1,\cdots,k-1$, denotes an eigenfunction of
$\mu_{i}(\Omega)$. Here, as usual, $W^{1,2}(\Omega)$ denotes a
Sobolev space which is the completion of the set of smooth functions
$C^{\infty}(\Omega)$ under the Sobolev norm $\|\cdot\|_{1,2}$
defined by (\ref{Sobo}).

Similar to (\ref{extre-1}), for bounded domains
$\Omega\subset\mathbb{R}^n$ with smooth boundary $\partial\Omega$,
$n\geq2$, the following extremum problem
 \begin{eqnarray} \label{extre-2}
 \max\left\{\mu_{k}(\Omega)\Big{|}|\Omega|_{n}=const.\right\}
\end{eqnarray}
can be asked for each $k=1,2,3,\cdots$. To the best of our
knowledge, for $k=1,2$, there exist affirmative answers to the
problem (\ref{extre-2}) as follows:
\begin{itemize}
\item (Szeg\H{o}-Weinberger
inequality, \cite{GS,HFW}) $\mu_{1}(\Omega)\leq\mu_{1}(B_{\Omega})$,
and the equality holds if and only if $\Omega$ is a ball in
$\mathbb{R}^n$ (which is congruent with $B_{\Omega}$,
$|B_{\Omega}|_{n}=|\Omega|_{n}=const.$). That is to say, among all
bounded domains in $\mathbb{R}^n$ having the same volume, Euclidean
balls maximize the first nonzero Neumann eigenvalue of the
Laplacian.

\item (Bucur-Henrot \cite{BH}) Let $\Omega\subset\mathbb{R}^n$ be a bounded open set such that the Sobolev
space $W^{1,2}(\Omega)$ is compactly embedded\footnote{~In fact, the
regularity that $\partial\Omega$ is Lipschitz continuous is
sufficient such that $W^{1,2}(\Omega)$ is compactly embedded in
$L^{2}(\Omega)$. Therefore, the smoothness assumption for the
boundary $\partial\Omega$ is much enough to investigate the maximum
of $\mu_{2}(\Omega)$ under the constraint of fixed volume.} in
$L^{2}(\Omega)$. Then
 \begin{eqnarray} \label{phy-1}
 |\Omega|_{n}^{2/n}\mu_{2}(\Omega)\leq2^{2/n}|B|_{n}^{2/n}\mu_{1}(B),
 \end{eqnarray}
  where $B$ is any ball in $\mathbb{R}^n$. If equality in (\ref{phy-1}) occurs, then $\Omega$ coincides a.e. with the union of two disjoint, equal
  balls. Clearly, the quantity $2^{2/n}|B|^{2/n}\mu_{1}(B)$ is scale
  invariant. Using (\ref{phy-1}) directly,  one has
  $\mu_{2}(\Omega)\leq2^{2/n}\mu_{1}(B_{\Omega})$, with a ball
  $B_{\Omega}$ satisfying
  $|B_{\Omega}|_{n}=|\Omega|_{n}=const.$, which gives an affirmative answer
  to the problem (\ref{extre-2}) for $k=2$.
\end{itemize}
For higher order ($k\geq3$) Neumann eigenvalues, not much is known.
However, recent years, some works have shown numerical approaches
which propose candidates for the optimizers for Dirichlet/Neumann
eigenvalues of the Laplacian and related spectral problems, and
which also suggest conjectures about their qualitative properties --
see, e.g., \cite{AF,BB,OB1} for details.

As mentioned above, in some situation, the eigenvalue problems
(\ref{eigen-D1}) and (\ref{eigen-N1})  have physical backgrounds,
and hence eigenvalues in discrete spectrum are called frequencies.
So, sometimes, spectral isoperimetric inequalities introduced above
are also called \emph{physical isoperimetric inequalities}. There is
also one more thing we wish to say here, that is, spectral
isoperimetric inequalities mentioned above hold may not only in
Euclidean spaces but also some curved spaces -- for instance, at
least one also has the Faber-Krahn inequality in hyperbolic spaces
and spheres. In fact, a more general version of Faber-Krahn
inequality says that (see, e.g., \cite[Chapter IV]{IC}):
\begin{itemize}
\item Let $\mathbb{M}^{n}(\kappa)$ be the complete, simply connected, $n$-dimensional ($n\geq2$) space form of constant sectional curvature
$\kappa$, and let $\mathbb{D}$ denote a geodesic disk in
$\mathbb{M}^{n}(\kappa)$. For a complete Riemannian $n$-manifold
$M^{n}$, $n\geq2$, and each open set $\Omega$, consisting of a
finite disjoint union of regular\footnote{~Here, following the
convention in \cite{IC}, ``\emph{regular}" means that the domain
considered has compact closure and smooth boundary, while the word
``\emph{normal}" also in this statement means that the domain
considered has compact closure and piecewise smooth boundary.}
domains in $M^{n}$, and satisfying
 \begin{eqnarray} \label{volI}
 |\Omega|_{n}=|\mathbb{D}|_{n}.
 \end{eqnarray}
  (If $\kappa>0$, then only consider those $\Omega$ for which
  $|\Omega|_{n}<|\mathbb{M}^{n}(\kappa)|_{n}$.) If, for all such $\Omega$ in
  $M^{n}$, equality (\ref{volI})  implies the geometric isoperimetric inequality
    \begin{eqnarray} \label{areaI}
  |\partial\Omega|_{n-1}\geq|\partial\mathbb{D}|_{n-1},
    \end{eqnarray}
    with equality in (\ref{areaI}) if and only if $\Omega$ is
    isometric to $\mathbb{D}$, then we also have, for
every normal domain $\Omega$ in $M^n$, that equality (\ref{volI})
implies the inequality
 \begin{eqnarray} \label{specI}
 \lambda_{1}(\Omega)\geq\lambda_{1}(\mathbb{D}),
 \end{eqnarray}
   with equality in (\ref{specI}) if and only if $\Omega$ is
    isometric to $\mathbb{D}$.
\end{itemize}
This fact can be simply summarized as ``\emph{under the constraint
of volume fixed, the geometric isoperimetric inequality
(\ref{areaI}) would imply the physical isoperimetric inequality
(\ref{specI})}". It is known that in space forms, (\ref{areaI})
holds once $|\Omega|_{n}=|\mathbb{D}|_{n}$. Hence, in space forms,
one has the physical isoperimetric inequality (\ref{specI}) under
the volume constraint (\ref{volI}). From this example, one might
have a recognition that geometric isoperimetric inequalities have a
close relation with physical isoperimetric inequalities (of
differential operators). A natural question is ``\emph{except space
forms, whether one could find other spaces on which the geometric
isoperimetric inequality (\ref{areaI}) holds under the volume
constraint (\ref{volI})?}". One might refer to \cite[Chapter IV]{IC}
for some interesting progresses on this question.

In the sequel, we will show a way to extend the Faber-Krahn
inequality, the Hong-Krahn-Szeg\H{o} inequality and the
Szeg\H{o}-Weinberger inequality of the Laplacian to the case of the
Witten-Laplacian.

For a given complete Riemannian $n$-manifold ($n\geq2$) with the
metric $g$, let $\Omega\subset M^{n}$ be a bounded
 domain (with boundary $\partial\Omega$) in $M^n$, and $\phi\in C^{\infty}(M^n)$ be a
smooth\footnote{~In fact, one might see that $\phi\in C^{2}(\Omega)$
is suitable to derive our main conclusions in this paper. However,
in order to avoid putting too much attention on discussion for the
regularity of the boundary $\partial\Omega$, and following the
assumption on conformal factor $e^{-\phi}$ for the notion of
\emph{smooth metric measure spaces} in many literatures (including
of course those cited in this paper), without specification, we wish
to assume that $\phi$ is smooth on the domain $\Omega$. }
real-valued function defined on $\Omega$. In this setting, one can
define the following elliptic operator
\begin{eqnarray*}
\Delta_{\phi}:=\Delta-\langle\nabla\phi,\nabla\cdot\rangle
\end{eqnarray*}
on $\Omega$, which is called the \emph{Witten-Laplacian} (also
called the \emph{drifting Laplacian} or the \emph{weighted
Laplacian}) w.r.t. the metric $g$. Consider the Dirichlet eigenvalue
problem of the Witten-Laplacian as follows
\begin{eqnarray} \label{eigen-D11}
\left\{
\begin{array}{ll}
\Delta_{\phi}u+\lambda u=0\qquad & \mathrm{in}~\Omega\subset M^{n}, \\[0.5mm]
u=0 \qquad & \mathrm{on}~{\partial\Omega},
\end{array}
\right.
\end{eqnarray}
and it is not hard to check that the operator $\Delta_{\phi}$ in
(\ref{eigen-D11}) is \textbf{self-adjoint} w.r.t. the following
inner product
\begin{eqnarray}  \label{inn-p}
(h_{1},h_{2})_{\phi}:=\int_{\Omega}h_{1}h_{2}d\eta=\int_{\Omega}h_{1}h_{2}e^{-\phi}dv,
\end{eqnarray}
with $h_{1},h_{2}\in W^{1,2}_{0,\phi}(\Omega)$,  where
$d\eta:=e^{-\phi}dv$ is the weighted measure, and
$W^{1,2}_{0,\phi}(\Omega)$ stands for a Sobolev space, which is the
completion of the set of smooth functions (with compact support)
$C^{\infty}_{0}(\Omega)$ under the following Sobolev norm
\begin{eqnarray}  \label{Sobo-1}
\|f\|^{\phi}_{1,2}:=\left(\int_{\Omega}f^{2}e^{-\phi}dv+\int_{\Omega}|\nabla
f|^{2}e^{-\phi}dv\right)^{1/2}=\left(\int_{\Omega}f^{2}d\eta+\int_{\Omega}|\nabla
f|^{2}d\eta\right)^{1/2}.
\end{eqnarray}
Then using similar arguments to those of the classical fixed
membrane problem of the Laplacian (i.e., the discussions about the
existence of discrete spectrum, Rayleigh's theorem, Max-min theorem,
etc. Those discussions are standard, and for details, please see for
instance \cite{IC}), it is not hard to know:
\begin{itemize}
\item The self-adjoint elliptic operator $-\Delta_{\phi}$ in
(\ref{eigen-D11}) \emph{only} has discrete spectrum, and all the
 eigenvalues in this discrete spectrum can be listed
non-decreasingly as follows
\begin{eqnarray} \label{sequence-3}
0<\lambda_{1,\phi}(\Omega)<\lambda_{2,\phi}(\Omega)\leq\lambda_{3,\phi}(\Omega)\leq\cdots\uparrow+\infty.
\end{eqnarray}
Each eigenvalue $\lambda_{i,\phi}$, $i=1,2,\cdots$, in the sequence
(\ref{sequence-3}) was repeated according to its multiplicity (which
is finite and equals to the dimension of the eigenspace of
$\lambda_{i,\phi}$). By applying the standard variational
principles, one can obtain that the $k$-th Dirichlet eigenvalue
$\lambda_{k,\phi}(\Omega)$ can be characterized as follows
 \begin{eqnarray}  \label{chr-1}
 \lambda_{k,\phi}(\Omega)=\inf\left\{\frac{\int_{\Omega}|\nabla f|^{2}e^{-\phi}dv}{\int_{\Omega}f^{2}e^{-\phi}dv}
 \Bigg{|}f\in W^{1,2}_{0,\phi}(\Omega),f\neq0,\int_{\Omega}ff_{i}e^{-\phi}dv=0\right\},
 \end{eqnarray}
where $f_{i}$, $i=1,2,\cdots,k-1$, denotes an eigenfunction of
$\lambda_{i,\phi}(\Omega)$. Moreover, the first Dirichlet eigenvalue
$\lambda_{1,\phi}(\Omega)$ of the eigenvalue problem
(\ref{eigen-D11}) satisfies
\begin{eqnarray}  \label{chr-2}
 \lambda_{1,\phi}(\Omega)=\inf\left\{\frac{\int_{\Omega}|\nabla f|^{2}d\eta}{\int_{\Omega}f^{2}d\eta}
 \Bigg{|}f\in W^{1,2}_{0,\phi}(\Omega),f\neq0\right\}.
 \end{eqnarray}
\end{itemize}
It is interesting and important to study spectral geometric problems
related to the Witten-Laplacian -- we refer to
\cite[Introduction]{CM1} for a detailed explanation. We already have
some interesting works about spectral estimates and geometric
functional inequalities related to the Witten-Laplacian -- see,
e.g., \cite{DMWW,LMWZ,JM1,JM2,MTZ,YWMD}.

On $\Omega$, one can also define a notion \emph{weighted volume} (or
$\phi$-volume) as follows:
\begin{eqnarray*}
|\Omega|_{n,\phi}:=\int_{\Omega}d\eta=\int_{\Omega}e^{-\phi}dv.
\end{eqnarray*}
Using the constraint of fixed weighted volume, we can obtain several
spectral isoperimetric inequalities for the first and the second
Dirichlet eigenvalues of the Witten-Laplacian. However, in order to
state our conclusions clearly, we need to impose an assumption on
the function $\phi$ as follows:
\begin{itemize}
\item (\textbf{Property 1}) Furthermore, $\phi$ is a function
of the Riemannian distance parameter $t:=d(o,\cdot)$ for some point
$o\in M^{n}$.
\end{itemize}
Clearly, if a given open Riemannian $n$-manifold $(M^{n},g)$ was
endowed with the weighted density $e^{-\phi}dv$, where $\phi$
satisfies \textbf{Property 1}, then $\phi$ would be a
\emph{\textbf{radial}} function defined on $M^{n}$ w.r.t. the radial
distance $t$, $t\in[0,\infty)$. Especially, when the given open
$n$-manifold is chosen to be $\mathbb{R}^{n}$ or $\mathbb{H}^{n}$
(i.e., the $n$-dimensional hyperbolic space of sectional curvature
$-1$), we additionally require that $o$ is the origin  of
$\mathbb{R}^{n}$ or $\mathbb{H}^{n}$.

First, we have the following Faber-Krahn type inequality for the
Witten-Laplcian in the Euclidean space.

\begin{theorem} \label{theo-1}
Assume that the function $\phi$ satisfies \textbf{Property 1} (with
$M^{n}$ chosen to be $\mathbb{R}^n$) and is concave. Let $\Omega$ be
a bounded domain with smooth boundary in
 $\mathbb{R}^n$, and let $B_{R}(o)$ be a ball of radius $R$ and centered at the
 origin $o$
 of
 $\mathbb{R}^{n}$ such that $|\Omega|_{n,\phi}=|B_{R}(o)|_{n,\phi}$,
 i.e. $\int_{\Omega}d\eta=\int_{B_{R}(o)}d\eta$. Then
 \begin{eqnarray*}
\lambda_{1,\phi}(\Omega)\geq\lambda_{1,\phi}(B_{R}(o)),
 \end{eqnarray*}
and the equality holds if and only if (up to measure zero) $\Omega$
is the ball $B_{R}(o)$, which lies entirely in the region
$B_{\mathcal{R}(h)}$ defined by (\ref{UIR-D}).
\end{theorem}

\begin{remark} \label{remark1-2}
\rm{ (1) Unlike the Neumann case described in Theorems \ref{theo-7}
and \ref{theo-8} below, for the Dirichlet case we do not need to
require that the point $o$ locates in the convex hull of the domain
$\Omega$ in Theorem \ref{theo-1}. The same situation also happens in
Theorem \ref{theo-2}.
 \\
(2) From the previous introduction on the Faber-Krahn inequality of
the Laplacian, one knows that under the volume constraint
(\ref{volI}), the geometric isoperimetric inequality (\ref{areaI})
makes an important role in the derivation process. What about the
Witten-Laplacian case? Does some weighted geometric isoperimetric
inequality play an important role also? The answer is affirmative.
We would like to recall a recent breakthrough of Chambers \cite{GRC}
to the Log-Convex Density Conjecture. Given a positive function $h$
in $\mathbb{R}^n$, $n\geq2$, one can define the weighted perimeter
and weighted volume of a set $A\subset\mathbb{R}^n$ of locally
finite perimeter as
 \begin{eqnarray*}
 \mathrm{Per}(A)=\int_{\partial A}hd\mathcal{H}^{n-1}, \qquad
 \mathrm{Vol}(A)=\int_{A}hd\mathcal{H}^{n},
 \end{eqnarray*}
where following the usage of notations in \cite{GRC},
$\mathcal{H}^m$ indicates the $m$-dimensional Hausdorff measure, and
$\partial A$ denotes the essential boundary of $A$. Such positive
function $h$ is called a \emph{density} on $\mathbb{R}^n$. If one
fixes a positive weighted volume $m>0$, does there exist a set
$A\subset\mathbb{R}^n$ such that $\mathrm{Vol}(A)=m$ and
 \begin{eqnarray*}
 \mathrm{Per}(A)=\inf\limits_{Q\subset\mathbb{R}^{n},\mathrm{Vol}(Q)=m}\mathrm{Per}(Q)?
 \end{eqnarray*}
Rosales, Ca\~{n}ete, Bayle and Morgan considered this problem and
gave a partial answer that in $\mathbb{R}^n$ with the density
$e^{c|x|^2}$, $c>0$, round balls about the origin uniquely minimize
perimeter for given volume (see \cite[Theorem 5.2]{RCBM}). Moreover,
they showed that for any radial, smooth density $h=e^{f(|x|)}$,
balls around the origin are stable\footnote{~Here ``\emph{stable}"
means that $\mathrm{Per}''(0)\geq0$ under smooth, volume-conserving
variations.} if and only $f$ is convex (\cite[Theorem 3.10]{RCBM}).
This fact motivates the following conjecture (3.12 in their
article), first stated by Kenneth Brakke:
\begin{itemize}
\item (Log-Convex Density Conjecture) \emph{In $\mathbb{R}^n$ with a smooth, radial, log-convex\footnote{~Clearly, for a density $h$ here, the log-convex assumption means $(\log h)''
\geq0$.} density, balls around the origin provide isoperimetric
regions of any given volume.}
\end{itemize}
Chambers \cite[Theorem 1.1]{GRC} gave an answer to the above
conjecture as follows:
 \begin{itemize}
 \item (\textbf{Fact A}) \emph{Given a density $h(x)=e^{f(|x|)}$ on $\mathbb{R}^n$ with $f$ smooth, convex and even, balls around the origin are isoperimetric regions with
respect to weighted perimeter and volume.}
 \end{itemize}
Moreover, Chambers \cite[Theorem 1.2]{GRC} characterized the
uniqueness of  isoperimetric regions as follows:
 \begin{itemize}
 \item (\textbf{Fact B})  \emph{Up to sets of measure $0$, the only
isoperimetric regions are balls centered at the origin, and balls
that lie entirely in
\begin{eqnarray} \label{UIR-D}
B_{\mathcal{R}(h)}=\{x\big{|}|x|\leq\mathcal{R}(h)\},
\end{eqnarray}
where $\mathcal{R}(h)=\sup\{|x|\big{|}h(x)=h(0)\}$.}
 \end{itemize}
\textbf{Fact A} and \textbf{Fact B} would make an important role in
the proof of Theorem \ref{theo-1} -- see Subsection \ref{s2-1} for
details. \\
 (3) Since Chambers' weighted geometric
isoperimetric inequality in $\mathbb{R}^n$ (i.e. \textbf{Fact A})
makes an important role in the proof of Theorem \ref{theo-1}, which
implies that similar to the potential precondition of \cite[Theorem
1.1]{GRC}, we also need to require that the boundary
$\partial\Omega$ has finite area (or following the convention in
\cite{GRC}, ``\emph{perimeter}") here. However, we think this
setting is so natural when considering the isoperimetric problems,
we wish not to
 list it out individually in every statement of our main conclusions
 in this paper. But, of course, $\partial\Omega$ should have this
 natural setting throughout the paper, which we do not mention again anymore.
 }
\end{remark}

We can prove the following:

\begin{theorem} \label{theo-2}
Let $\mathbb{S}^{n}_{+}$ be an $n$-dimensional hemisphere of radius
$1$, and let $\Omega\subset\mathbb{S}^{n}_{+}$ be a bounded domain
whose boundary $\partial\Omega$ has positive constant mean
curvature. Assume that the function $\phi$ satisfies
\textbf{Property 1} (with $M^{n}$ chosen to be $\mathbb{S}^{n}_{+}$)
and moreover $\phi=-\log\cos t$, where the point $o$ mentioned in
\textbf{Property 1} should additionally be required to be the base
point of $\mathbb{S}^{n}_{+}$. Then
 \begin{eqnarray*}
\lambda_{1,\phi}(\Omega)\geq\lambda_{1,\phi}(B_{R}(o)),
 \end{eqnarray*}
  where $B_{R}(o)$ denotes a geodesic ball of radius $R$ and centered at the
 base point $o$
 of
 $\mathbb{S}^{n}_{+}$ such that $|\Omega|_{n,\phi}=|B_{R}(o)|_{n,\phi}$.
 The equality holds if and only if $\Omega$ is isometric to the geodesic
 ball $B_{R}(o)$.
\end{theorem}

\begin{remark}
\rm{ (1) When investigating the above Faber-Krahn type isoperimetric
inequality, there is no essential difference between
$\mathbb{S}^{n}_{+}$ and
 a hemisphere with radius not equal to $1$. \\
 (2) In order to let readers who might not know the concept ``\emph{the base point}" clearly, we wish to give an explanation here. It is better to
 start the explanation with spherically symmetric manifolds, which is also called generalized space forms (suggested in the work of Katz-Kondo \cite{KK}). We refer
 readers to \cite{FMI,M1,MDW} for a detailed description about the accurate definition, the basic properties and some interesting
 applications of spherically symmetric manifolds. The corresponding
 author has used spherically symmetric manifolds as the model space
 to derive some interesting comparison theorems (for volume, eigenvalues of different types, heat kernel, and some other geometric quantities) -- see, e.g.,
 \cite{FMI,JM4,JM3,YWMD}. In fact, one has:
 \begin{itemize}
 \item (\cite[Definition 2.1]{FMI}) For a given complete $n$-manifold $M^{n}$, a domain $\mathcal{D}=\exp_{p}([0,l)\times S_{p}^{n-1})\subset
 M^{n}\setminus Cut(p)$, with $l<inj(p)$, is said
to be spherically symmetric with respect to a point
$p\in\mathcal{D}$, if and only if the matrix $\mathbb{A}(t,\xi)$
satisfies $\mathbb{A}(t,\xi)=f(t)I$, for a function $f\in
C^{2}([0,l))$ with $f(0)=0$, $f'(0)=1$ and $f|_{(0,l)}>0$.
 \end{itemize}
 Here $S_{p}^{n-1}$ denotes the unit sphere of the tangent space $T_{p}M^{n}$, $Cut(p)$ stands for the cut-locus of the point $p$, $inj(p)$
 denotes the injectivity radius at $p$, $\xi\in S_{p}^{n-1}$, and $\mathbb{A}(t,\xi):\xi^{\perp}\rightarrow\xi^{\perp}$ is the path of linear transformations well-defined
  in \cite[Section 2]{FMI}. A standard model for spherically
  symmetric manifolds is given by the quotient of the warped product
  $[0,l)\times_{f}\mathbb{S}^{n-1}$ with the metric
   \begin{eqnarray*}
 ds^{2}=dt^{2}+f^{2}(t)|d\xi|^2,\qquad \forall \xi\in
 S^{n-1}_{p},~0<t<l,
   \end{eqnarray*}
    where usually $|d\xi|^2$ denotes the round metric of the unit
    $(n-1)$-sphere $\mathbb{S}^{n-1}$. In this model, all pairs
    $(0,\xi)$ are identified with the single point $p$, which is
    called \emph{the base point} of the spherically symmetric domain
    $\mathcal{D}=[0,l)\times_{f}\mathbb{S}^{n-1}$. Clearly, as already revealed  in (2.12) of \cite{FMI}, a space form with constant sectional curvature $\kappa$ is
also a spherically symmetric manifold and in this particular
situation the warping function $f$ satisfies
  \begin{eqnarray*}
f(t)=\left\{
\begin{array}{lll}
\frac{\sin(\sqrt{\kappa}t)}{\sqrt{\kappa}},\qquad & l=\frac{\pi}{\sqrt{\kappa}}~~\quad\kappa>0,  \\[0.5mm]
t, \qquad & l=+\infty \quad\kappa=0,  \\[0.5mm]
\frac{\sinh(\sqrt{-\kappa}t)}{\sqrt{-\kappa}},\qquad &
l=+\infty\quad\kappa<0.
\end{array}
\right.
  \end{eqnarray*}
 (3) Since $o$ is required to be the base point of
 $\mathbb{S}^{n}_{+}$, then for the domain
 $\Omega\subset\mathbb{S}^{n}_{+}$ in Theorem \ref{theo-2}, the range of the Riemannian distance parameter
 $t=d(o,\cdot)$ should be $(0,\pi/2)$, which implies that the choice of
 the function $\phi=-\log\cos t$ makes sense. Besides, in fact,
 $\mathbb{S}^{n}_{+}$ can be modeled as $[0,\pi/2]\times_{\sin
 t}\mathbb{S}^{n-1}$ with the metric $dt^{2}+(\sin
 t)^{2}|d\xi|^{2}$, and its base point $o$ should be the vertex of
 $\mathbb{S}^{n}_{+}$.
}
\end{remark}

We can also get the following:

\begin{theorem} \label{theo-3}
 Assume that
the function $\phi$ satisfies \textbf{Property 1} (with $M^{n}$
chosen to be $\mathbb{H}^{n}$) and is strictly concave,
 where the point $o$
mentioned in \textbf{Property 1} should additionally be required to
be the origin of $\mathbb{H}^{n}$. Let $\Omega\subset\mathbb{H}^{n}$
be a bounded domain with boundary. Then
 \begin{eqnarray*}
\lambda_{1,\phi}(\Omega)\geq\lambda_{1,\phi}(B_{R}(o)),
 \end{eqnarray*}
  where $B_{R}(o)$ denotes a geodesic ball of radius $R$ and centered at the
 origin $o$
 of
 $\mathbb{H}^{n}$ such that $|\Omega|_{n,\phi}=|B_{R}(o)|_{n,\phi}$.
 The equality holds if and only if $\Omega$ is isometric to the geodesic
 ball $B_{R}(o)$.
\end{theorem}

\begin{remark}
\rm{ (1) The hyperbolic space $\mathbb{H}^{n}$ can be modeled as
$[0,\infty)\times_{\sinh t}\mathbb{S}^{n-1}$ with the metric
 \begin{eqnarray*}
dt^{2}+(\sinh t)^{2}|d\xi|^{2}.
 \end{eqnarray*}
Since hyperbolic spaces are two-point homogenous, the base point of
$\mathbb{H}^n$ is not unique and any point of $\mathbb{H}^n$ can be
chosen as the base point, which is different with the case of
hemisphere $\mathbb{S}^{n}_{+}$. However, for $\mathbb{H}^n$ once
its globally defined coordinate system was set up, the origin $o$
would be determined uniquely w.r.t. this system. As shown above,  in
order to get the main conclusion in Theorem \ref{theo-3}, we need to
assume that $\phi$ is radial w.r.t. some fixed point and is also
concave, which leads to the situation that in the statement of
Theorem \ref{theo-3}, it is better to choose the point $o$ to be the
origin of $\mathbb{H}^n$ (might not the base point), and
correspondingly $\phi$ is concave w.r.t. the radial Riemannian
distance parameter $t=d(o,\cdot)$.
 \\
 (2) As mentioned before, one knows two facts: (a) under the constraint of
 fixed volume, the Faber-Krahn inequality for the first Dirichlet
 eigenvalue of the Laplacian also holds in hyperbolic spaces; (b)
 under the constraint of fixed weighted volume,
\textbf{Fact A} (i.e., a weighted geometric isoperimetric inequality
in  $\mathbb{R}^n$) makes an important role in the proof of the
Faber-Krahn type inequality for the Witten-Laplcian in
$\mathbb{R}^n$ (i.e. Theorem \ref{theo-1}). So, it is natural to
ask:
 \begin{itemize}
 \item \emph{Could one expect to get a hyperbolic version of \textbf{Fact A} which makes a contribution in the proof of Theorem \ref{theo-3}?}
 \end{itemize}
The answer is affirmative. In fact, Li-Xu \cite[Theorem 1.1]{LX}
obtained a partial result to the hyperbolic version of \textbf{Fact
A} for specified density through suitably applying Chambers' result
\cite{GRC} by projecting the hyperbolic space onto $\mathbb{R}^n$
and employing a comparison argument. Very recently, L. Silini
\cite{LS} solved the above question completely. For an arbitrary
base point $o\in \mathbb{H}^{n}$, and a density $h$ given by
$h:=e^{f(d(o,\cdot))}$, where $h:\mathbb{R}\rightarrow\mathbb{R}$ is
a smooth, (strictly) convex, even function, and, similar as before,
$d(o,\cdot)$ denotes the Riemannian distance to the point $o$ on
$\mathbb{H}^{n}$, one can define the weighted perimeter and weighted
volume of a set with finite perimeter $E\subset \mathbb{H}^{n}$ as
follows
\begin{eqnarray*}
 \mathrm{P}_{h}(E)=\int_{\partial^{\ast}E}hd\mathcal{H}^{n-1}, \qquad
 \mathrm{V}_{h}(E)=\int_{E}hd\mathcal{H}^{n},
 \end{eqnarray*}
where following the usage of notations in \cite{LS},
$\partial^{\ast}E$ denotes the reduced boundary of $E$, and
$\mathcal{H}^m$ indicates the $m$-dimensional Hausdorff measure.
Silini \cite[Theorem 1.1]{LS} proved the following:
\begin{itemize}
 \item (\textbf{Fact C}) \emph{For any strictly radially log-convex density $h$, geodesic balls centered
 at $o\in \mathbb{H}^{n}$ uniquely minimize the weighted perimeter for any given weighted volume with respect
to $P_{h}$ and $V_{h}$.}
\end{itemize}
\textbf{Fact C} would make an important role in the proof of Theorem
\ref{theo-3} -- see Section \ref{s3} for details. Using a comparison
argument between $H^{n}_{\mathbb{C}}=U(n,1)/U(n)$ (i.e. the
$n$-dimensional complex hyperbolic space of constant curvature $-1$)
and $\mathbb{H}^{2n}$, together with \textbf{Fact C}, Silini
\cite{LS} can get further:
 \begin{itemize}
 \item In $H^{n}_{\mathbb{C}}$,
geodesic balls are uniquely isoperimetric in the class of
Hopf-symmetric sets for all volumes.
 \end{itemize}
 This conclusion gives a partial answer to an open conjecture
 proposed by Gromov-Ros in \cite{GKPS} as follows:
 \begin{itemize}
\item (\textbf{Conjecture}) \emph{Geodesic balls are isoperimetric for all
volumes in the complex hyperbolic space $H^{n}_{\mathbb{C}}$.}
 \end{itemize}
Silini's above result on the isoperimetric problem for the class of
Hopf-symmetric sets in  $H^{n}_{\mathbb{C}}$ might inspire readers to try to extend the spectral isoperimetric inequality in
Theorem \ref{theo-3} to a more general space, which we think it is possible. However, due to the structure of this paper, here we just focus
on investigating spectral isoperimetric inequalities for the Witten-Laplacian on bounded domains in space forms.\\
 (3) As explained in \cite[Remark 1.7]{LS}, since technical difficulties
 arise from the presence of regions with constant weight, for simplicity it
was decided to  to assume the weight to be strictly log-convex
rather than simply log-convex in extending the proof of Brakke's
conjecture from the Euclidean space to the hyperbolic space. This is
the reason why in Theorem \ref{theo-3} we assume that the radial
function $\phi$ is strictly concave (i.e., $-(\log \phi)''>0$).
Besides, if the domain $\Omega$ has a constant weight (i.e., a
constant density), then the Witten-Laplacian degenerates into the
classical Laplacian, and correspondingly, in $\mathbb{H}^n$ one
naturally has the Faber-Krahn inequality for the first Dirichlet
eigenvalue. In this situation, it is no need to write down Theorem
\ref{theo-3} any more. Based on this truth, in Theorem \ref{theo-3}
it is acceptable to assume that the radial function $\phi$ is
strictly concave.
 }
\end{remark}

Inspired by the technique used in \cite{BBMP}, under other
assumptions on $\phi$ and the constraint of weighted volume fixed,
we can also get the following Faber-Krahn type inequality for the
Witten-Laplcian in the Euclidean space, which can be seen as a
complement to Theorem \ref{theo-1}.

\begin{theorem}  \label{theo-4}
Assume that the function $\phi$ satisfies \textbf{Property 1} (with
$M^{n}$ chosen to be $\mathbb{R}^n$), $\phi$ is monotone
non-increasing, and for $z\geq0$, the function
 \begin{eqnarray*}
\left(e^{-\phi(z^{\frac{1}{n}})}-e^{-\phi(0)}\right)z^{1-\frac{1}{n}}
 \end{eqnarray*}
is convex. Let $\Omega$ be a bounded domain with Lipschitz boundary
in $\mathbb{R}^n$, and let $B_{R}(o)$ be a ball of radius $R$ and
centered at the
 origin $o$
 of
 $\mathbb{R}^{n}$ such that $|\Omega|_{n,\phi}=|B_{R}(o)|_{n,\phi}$. Then
 \begin{eqnarray*}
\lambda_{1,\phi}(\Omega)\geq\lambda_{1,\phi}(B_{R}(o)).
 \end{eqnarray*}
\end{theorem}

\begin{remark}
\rm{Since $\phi$ satisfies \textbf{Property 1} and moreover when
$M^n$ is chosen to be $\mathbb{R}^n$, we additionally require that
$o$ is the origin  of $\mathbb{R}^{n}$, so $o$ corresponds to $z=0$,
and then $\phi(0)$ is actually the value of the function $\phi$ at
the origin $o$.}
\end{remark}

For the second Dirichlet eigenvalue of the Witten-Laplacian, we can
obtain the following Hong-Krahn-Szeg\H{o} type inequalities.

\begin{theorem}  \label{theo-5}
Assume that the function $\phi$ satisfies \textbf{Property 1} (with
$M^{n}$ chosen to be $\mathbb{R}^n$) and is concave. Let $\Omega$ be
a bounded domain with smooth boundary in
 $\mathbb{R}^n$, and let $B_{\widetilde{R}}(o)$ be a ball of radius $\widetilde{R}$ and centered at the
 origin $o$
 of
 $\mathbb{R}^{n}$ such that $|\Omega|_{n,\phi}/2=|B_{\widetilde{R}}(o)|_{n,\phi}$,
 i.e. $\frac{1}{2}\int_{\Omega}d\eta=\int_{B_{\widetilde{R}}(o)}d\eta$. Then
 \begin{eqnarray*}
\lambda_{2,\phi}(\Omega)\geq\lambda_{1,\phi}(B_{\widetilde{R}}(o)).
 \end{eqnarray*}
 That is to say, under the assumptions for $\phi$ described above, the
minimum of the second Dirichlet eigenvalue of the Witten-Laplacian
on bounded domains $\Omega$ in $\mathbb{R}^n$, whose weighted volume
equals some prescribed positive constant, should be equal to the
first Dirichlet eigenvalue of the Witten-Laplacian on a ball
$B_{\widetilde{R}}(o)$ (of radius $\widetilde{R}$ and centered at
the
 origin $o\in\mathbb{R}^n$) such that
$|\Omega|_{n,\phi}/2=|B_{\widetilde{R}}(o)|_{n,\phi}$.
\end{theorem}

\begin{theorem}  \label{theo-6}
 Assume that the function $\phi$ satisfies \textbf{Property
1} (with $M^{n}$ chosen to be $\mathbb{H}^{n}$) and is strictly
concave,
 where the point $o$
mentioned in \textbf{Property 1} should additionally be required to
the origin of $\mathbb{H}^{n}$. Let $\Omega\subset\mathbb{H}^{n}$ be
a bounded domain with boundary. Then
 \begin{eqnarray*}
\lambda_{2,\phi}(\Omega)\geq\lambda_{1,\phi}(B_{\widetilde{R}}(o)),
 \end{eqnarray*}
  where $B_{\widetilde{R}}(o)$ denotes a geodesic ball of radius $\widetilde{R}$ and centered at the
 origin $o$
 of
 $\mathbb{H}^{n}$ such that
 $|\Omega|_{n,\phi}/2=|B_{\widetilde{R}}(o)|_{n,\phi}$. That is to say, under the assumptions for $\phi$ described above, the
minimum of the second Dirichlet eigenvalue of the Witten-Laplacian
on bounded domains $\Omega$ in $\mathbb{H}^n$, whose weighted volume
equals some prescribed positive constant, should be equal to the
first Dirichlet eigenvalue of the Witten-Laplacian on a geodesic
ball $B_{\widetilde{R}}(o)$ (of radius $\widetilde{R}$ and centered
at the
 origin $o\in\mathbb{H}^n$) such that
$|\Omega|_{n,\phi}/2=|B_{\widetilde{R}}(o)|_{n,\phi}$.
\end{theorem}

For a bounded domain $\Omega$ (with boundary $\partial\Omega$) on a
given $n$-dimensional ($n\geq2$) complete Riemannian manifold $M^n$,
we can also consider the following Neumann eigenvalue problem of the
Witten-Laplacian
\begin{eqnarray} \label{eigen-N11}
\left\{
\begin{array}{ll}
\Delta_{\phi} u+\mu u=0\qquad & \mathrm{in}~\Omega\subset M^{n}, \\[0.5mm]
\frac{\partial u}{\partial\vec{\nu}}=0 \qquad &
\mathrm{on}~{\partial\Omega},
\end{array}
\right.
\end{eqnarray}
and it is easy to check that the operator $\Delta_{\phi}$ in
(\ref{eigen-N11}) is \textbf{self-adjoint} w.r.t. the inner product
(\ref{inn-p}) with $h_{1},h_{2}\in W_{\phi}^{1,2}(\Omega)$. Here
$W_{\phi}^{1,2}(\Omega)$ stands for a Sobolev space, which is the
completion of the set of smooth functions $C^{\infty}(\Omega)$ under
the Sobolev norm $\|\cdot\|^{\phi}_{1,2}$ defined by (\ref{Sobo-1}).
Then using similar arguments to those of the classical free membrane
problem of the Laplacian (see, e.g., \cite{IC}), it is not hard to
know:
 \begin{itemize}
\item  The operator $-\Delta_{\phi}$ in
(\ref{eigen-N11}) \emph{only} has discrete spectrum, and all the
 eigenvalues in this discrete spectrum can be listed
non-decreasingly as follows
\begin{eqnarray} \label{sequence-4}
0=\mu_{0,\phi}(\Omega)<\mu_{1,\phi}(\Omega)\leq\mu_{2,\phi}(\Omega)\leq\mu_{3,\phi}(\Omega)\leq\cdots\uparrow+\infty.
\end{eqnarray}
Each eigenvalue $\mu_{i,\phi}$, $i=0,1,2,\cdots$, in the sequence
(\ref{sequence-4}) is repeated according to its multiplicity (i.e.,
the dimension of the eigenspace of $\mu_{i,\phi}$). Specially, the
zero eigenvalue $\mu_{0,\phi}$ has multiplicity $1$ and has nonzero
constant function as its eigenfunction. By applying the standard
variational principles, one can obtain that the $k$-th Neumann
eigenvalue $\mu_{k,\phi}(\Omega)$ can be characterized as follows
 \begin{eqnarray}  \label{chr-3}
 \mu_{k,\phi}(\Omega)=\inf\left\{\frac{\int_{\Omega}|\nabla f|^{2}e^{-\phi}dv}{\int_{\Omega}f^{2}e^{-\phi}dv}
 \Bigg{|}f\in W_{\phi}^{1,2}(\Omega),f\neq0,\int_{\Omega}ff_{i}e^{-\phi}dv=0\right\},
 \end{eqnarray}
where $f_{i}$, $i=1,2,\cdots,k-1$, denotes an eigenfunction of
$\mu_{i,\phi}(\Omega)$. Moreover, the first nonzero Neumann
eigenvalue $\mu_{1,\phi}(\Omega)$ of the eigenvalue problem
(\ref{eigen-N11}) satisfies
\begin{eqnarray}  \label{chr-4}
 \mu_{1,\phi}(\Omega)=\inf\left\{\frac{\int_{\Omega}|\nabla f|^{2}d\eta}{\int_{\Omega}f^{2}d\eta}
 \Bigg{|}f\in W_{\phi}^{1,2}(\Omega),f\neq0,\int_{\Omega}fd\eta=0\right\}.
 \end{eqnarray}
 \end{itemize}
In fact, the above facts have been explained more clearly in
\cite[Section 1]{CM1}. Here we wish to keep writing down the above
content for two reasons: the one is for the completion of the brief
introduction to the eigenvalue problem (\ref{eigen-N11}) here; the
other one is that the characterization (\ref{chr-4}) would be used
to derive spectral isoperimetric inequalities for the first nonzero
Neumann eigenvalue $\mu_{1,\phi}(\cdot)$ below.

 We can prove the following
Szeg\H{o}-Weinberger type inequalities for the Witten-Laplacian.

\begin{theorem} \label{theo-7}
Let $\Omega$ be a bounded domain with smooth boundary in
 $\mathbb{R}^n$. Assume that the function $\phi$ satisfies \textbf{Property 1} (with
$M^{n}$ chosen to be $\mathbb{R}^n$ and additionally the point $o$
required to be in the convex hull of $\Omega$, i.e.
$o\in{\mathrm{hull}}(\Omega)$), and $\phi$ is also a non-increasing
convex function defined on $[0,\infty)$. Let $B_{R}(o)$ be a ball of
radius $R$ and centered at the
 origin $o$
 of
 $\mathbb{R}^{n}$ such that $|\Omega|_{n,\phi}=|B_{R}(o)|_{n,\phi}$,
 i.e. $\int_{\Omega}d\eta=\int_{B_{R}(o)}d\eta$.  Then
 \begin{eqnarray*}
\mu_{1,\phi}(\Omega)\leq\mu_{1,\phi}(B_{R}(o)),
\end{eqnarray*}
with equality holding if and only if $\Omega$ is the ball
$B_{R}(o)$.
\end{theorem}

\begin{theorem}  \label{theo-8}
Let $\Omega$ be a bounded domain with smooth boundary in
 $\mathbb{H}^n$. Assume that the function $\phi$ satisfies \textbf{Property 1} (with
$M^{n}$ chosen to be $\mathbb{H}^n$ and additionally
$o\in{\mathrm{hull}}(\Omega)$), and $\phi$ is also a non-increasing
convex function defined on $[0,\infty)$. Let $B_{R}(o)$ be a
geodesic ball of radius $R$ and centered at the
 origin $o$
 of
 $\mathbb{H}^{n}$ such that $|\Omega|_{n,\phi}=|B_{R}(o)|_{n,\phi}$.  Then
 \begin{eqnarray*}
\mu_{1,\phi}(\Omega)\leq\mu_{1,\phi}(B_{R}(o)),
\end{eqnarray*}
with equality holding if and only if $\Omega$ is isometric to the
geodesic ball $B_{R}(o)$.
\end{theorem}

\begin{remark}
\rm{ (1) In fact, in our very recent work \cite[Theorems 1.1 and
1.5]{CM1}, we can prove an isoperimetric inequality for the sums of
the reciprocals of the first $(n-1)$ nonzero Neumann eigenvalues of
the Witten-Laplacian on bounded domains in $\mathbb{R}^n$ or
$\mathbb{H}^n$, which together with the monotonicity of the sequence
(\ref{sequence-4}) of Neumann eigenvalues yields directly our
Theorem \ref{theo-7} and Theorem \ref{theo-8} here. This fact has
been already pointed out in \cite[Corollaries 1.2 and 1.6]{CM1},
and readers can check there for details. \\
(2) Based on two reasons, we insist on writing down Theorem
\ref{theo-7} and Theorem \ref{theo-8} here. The one is for the
completion of the whole structure of this paper, and the other one
is that our approach here for proving Theorem \ref{theo-7} and
Theorem
\ref{theo-8} is somehow different from the one used in \cite{CM1}.\\
(3) Different with the Dirichlet case, we need to require that
$o\in{\mathrm{hull}}(\Omega)$ in Theorem \ref{theo-7} and Theorem
\ref{theo-8}. This is because we have to use the Brouwer fixed point
theorem to make sure the existence of an orthonormal frame field
such that the origin of the coordinate system (corresponding to the
orthonormal frame field) locates in the convex hull of $\Omega$, and
then all the computations involved trail functions constructed are
valid. See the proofs of Theorem \ref{theo-7} and Theorem
\ref{theo-8} in Section \ref{s3} for details. }
\end{remark}

The paper is organized as follows. The proofs of the Faber-Krahn
type inequalities, the Hong-Krahn-Szeg\H{o} type inequalities and
the Szeg\H{o}-Weinberger type inequalities for the Witten-Laplcian
will be given in Sections \ref{s2},  \ref{s3} and  \ref{s4}
respectively. Besides, in Section \ref{s5}, we will give the
detailed information about the first nonzero Neumann eigenvalue and
its eigenfunctions of the Witten-Laplacian on prescribed (geodesic)
balls in space forms.

\section{The Faber-Krahn type inequalities for the
Witten-Laplacian} \label{s2}
\renewcommand{\thesection}{\arabic{section}}
\renewcommand{\theequation}{\thesection.\arabic{equation}}
\setcounter{equation}{0}

\subsection{The Euclidean case} \label{s2-1}

Assume that  $f$ is an eigenfunction corresponding to the first
Dirichlet eigenvalue $\lambda_{1,\phi}(\Omega)$. Since $f$ does not
change sign on $\Omega$, without loss of generality, we can assume
$f>0$ on $\Omega$ (see Lemma \ref{lemma3-1} below for the
explanation). Consider the sets $\Omega_{s}:=\{x\in\Omega|f(x)>s\}$,
and let $\Omega^{\ast}_{s}$ be balls in $\mathbb{R}^n$ with center
at the origin $o$ and satisfying
$|\Omega_{s}|_{n,\phi}=|\Omega_{s}^{\ast}|_{n,\phi}$. Let $B_{R}(o)$
be a ball of radius $R$ and centered at $o$
 of
 $\mathbb{R}^{n}$ such that $|\Omega|_{n,\phi}=|B_{R}(o)|_{n,\phi}$,
 i.e. $\int_{\Omega}d\eta=\int_{B_{R}(o)}d\eta$. Define a function
 $f^{\ast}$ on $B_{R}(o)$ having the following properties:
 \begin{itemize}
 \item $f^{\ast}$ is a radial decreasing function;
 \item $f^{\ast}$ takes the value $s$ on the boundary sphere
 $\partial\Omega_{s}^{\ast}$ of the ball $\Omega^{\ast}_{s}$ (for a fixed $s$).
 \end{itemize}
 It is not hard to see that $\Omega_{0}=\Omega$ and correspondingly $\Omega_{0}^{\ast}=B_{R}(o)$.
 The existence of the balls $\Omega_{s}^{\ast}$ can be assured by
 using the Schwarz symmetrization. Readers can check e.g. \cite{BBMP, AH1} for details on how to use symmetrization to
 get balls $\Omega_{s}^{\ast}$ under the constraint of having the same
 weighted volume.

Now, we make an agreement on the notations used right below. Denote
by $\widehat{dv}$ the $(n-1)$-dimensional Hausdorff measure of the
boundary associated to the Riemannian volume element\footnote{~In
fact, for domains $\Omega_{s}$ and $\Omega=\Omega_{0}$, they should
have the same volume element $dv$. However, in order to emphasize
that the domain $\Omega_{s}$ depends on $s$, we wish to additionally
write the volume element of $\Omega_{s}$ as $dv_{s}$ (except
$s=0$).} $dv$, and this convention will be used throughout the
paper. Similarly, $\widehat{d\eta}=e^{-\phi}\widehat{dv}$ would be
the weighted volume element of the boundary. Besides, for
convenience, set $G(s):=\partial\Omega_{s}$,
$S_{t(s)}:=(G(s))^{\ast}=G^{\ast}(s)=\partial\Omega_{s}^{\ast}$
which denotes the sphere with center at the origin and radius
$t(s)$. The following formula is known as the co-area formula (see,
e.g., \cite{FEB1, IC}):
 \begin{itemize}
\item For any continuous function $h$ defined on $\Omega$, one has
 \begin{eqnarray} \label{2-1}
  \int_\Omega{h}dv=\int_0^{\sup f}\int_{G(s)}h|\nabla
f|^{-1}\widehat{dv_s} ds,
\end{eqnarray}
where following the above agreement $\widehat{dv_s}$ denotes the
volume element of the hypersurface $G(s)=f^{-1}(s)$.
\end{itemize}
Clearly, taking $h=|\nabla f|^2$ and then applying the co-area
formula, one has
 \begin{eqnarray*}
\int_\Omega{|\nabla f|^2}dv=\int_0^{\sup f}\int_{G(s)}|\nabla
f|\widehat{dv_s}ds.
 \end{eqnarray*}
Denote by the Schwarz symmetric rearrangement mapping $t:[0,\sup
f]\rightarrow[0,R]$, with $R$ the radius of $B_{R}(o)$, and $\psi$
 the inverse transformation of $t$, where $t$ additionally satisfies $t(0)=R$, $t(\sup
f)=0$.

\begin{lemma} \label{lemma2-1}
If $\Omega$ is a bounded region in $\mathbb{R}^n$, and $\phi$
satisfies \textbf{Property 1} (with $M^{n}$ chosen to be
$\mathbb{R}^n$), then
\begin{eqnarray} \label{2-2}
\int_{\Omega}f^{2}d\eta=\int_{B_{R}(o)}(f^{\ast})^2d\eta,
\end{eqnarray}
 where $B_{R}(o)\subset\mathbb{R}^n$ is the ball defined as in Theorem \ref{theo-1}.
\end{lemma}

\begin{proof}
By a direct calculation, one can obtain
\begin{eqnarray*}
\int_{B_{R}(o)}(f^{\ast})^2d\eta&=&\int_0^{R}\int_{\partial B_{t}(o)}(f^{\ast})^2 e^{-\phi(t)}\widehat{dv_{t}}dt \nonumber\\
&=&\int_0^{R}\psi^2(t)\int_{\partial B_{t}(o)}e^{-\phi(t)}\widehat{dv_{t}}dt \nonumber\\
&=&-\int_0^{\sup f}\psi^2(t(s))t'(s)\left(\int_{\partial B_{t(s)}(o)}e^{-\phi(t(s))}\widehat{dv_{t}}\right)ds \nonumber\\
&=&-\int_0^{\sup f}s^{2}\left(-\int_{G(s)}|\nabla f|^{-1}e^{-\phi|_{G(s)}}\widehat{dv_{s}}\right)ds \nonumber\\
&=&\int_{\Omega}f^2d\eta,
\end{eqnarray*}
 which implies (\ref{2-2}) directly.
\end{proof}

Now, together with \textbf{Fact A} and \textbf{Fact B}, we can get:

\begin{proof} [Proof of Theorem \ref{theo-1}]  Applying the
co-area formula, we have
\begin{eqnarray} \label{2-3}
\int_\Omega{|\nabla f|^2}e^{-\phi}dv=\int_0^{\sup
f}\int_{G(s)}|\nabla f|e^{-\phi}\widehat{dv_s}ds.
\end{eqnarray}
We can obtain by using the Cauchy-Schwarz inequality that
\begin{eqnarray}  \label{2-4}
\int_{G(s)}|\nabla f|e^{-\phi|_{G(s)}}\widehat{dv_{s}}\geq
\frac{\left(\int_{G(s)}e^{-\phi|_{G(s)}}\widehat{dv_{s}}\right)^2}{\int_{G(s)}|\nabla
f|^{-1}e^{-\phi|_{G(s)}}\widehat{dv_{s}}}.
\end{eqnarray}
By \textbf{Fact A} and \textbf{Fact B}, we have
$\int_{G(s)}e^{-\phi|_{G(s)}}\widehat{dv_{s}}\geq
\int_{G^{\ast}(s)}e^{-\phi(t(s))}\widehat{dv_{s}}$, with equality
holding if and only if $G(s)\setminus E(s)=G^{\ast}(s)$, where the
set $E(s)$ denotes a set of measure zero. Substituting this fact
into (\ref{2-4}) yields
\begin{eqnarray}  \label{2-5}
\int_{G(s)}|\nabla f|e^{-\phi|_{G(s)}}\widehat{dv_{s}}\geq
\frac{\left(\int_{G^{\ast}(s)}e^{-\phi(t(s))}\widehat{dv_{s}}\right)^2}{\int_{G(s)}|\nabla
f|^{-1}e^{-\phi|_{G(s)}}\widehat{dv_{s}}}.
\end{eqnarray}
On the other hand, one has
\begin{eqnarray*} \label{2-6}
\int_{G^{\ast}(s)}|\nabla f^{\ast}|e^{-\phi(t(s))}\widehat{dv_{s}}=
\frac{\left(\int_{G^{\ast}(s)}e^{-\phi(t(s))}\widehat{dv_{s}}\right)^2}{\int_{G^{\ast}(s)}|\nabla
f^{\ast}|^{-1}e^{-\phi(t(s))}\widehat{dv_{s}}},
\end{eqnarray*}
 since $|\nabla f^{\ast}|$ and $e^{-\phi(s)}$ are constant on the sphere $G^{\ast}(s)$. We notice that
 \begin{eqnarray*}
|\Omega_r|_{n,\phi}&=&\int_{\Omega_r}e^{-\phi}dv\\
&=&\int_r^{\sup f}\int_{G(s)}|\nabla
f|^{-1}e^{-\phi|_{G(s)}}\widehat{dv_{s}}ds,
 \end{eqnarray*}
and so it follows that
 \begin{eqnarray*} \label{2-7}
 \left(|\Omega_r|_{n,\phi}\right)'(s)=-\int_{G(s)}|\nabla
 f|^{-1}e^{-\phi|_{G(s)}}\widehat{dv_{s}},
 \end{eqnarray*}
which implies
 \begin{eqnarray} \label{2-8}
-\int_{G(s)}|\nabla
 f|^{-1}e^{-\phi|_{G(s)}}\widehat{dv_{s}}=\frac{d}{ds}|\Omega_s|_{n,\phi}=\frac{d}{ds}|\Omega^{\ast}_s|_{n,\phi}.
 \end{eqnarray}
Since
 \begin{eqnarray*}
|\Omega_s^{\ast}|_{n,\phi}=\int_0^{t(s)}\int_{\partial
B_{z}(o)}e^{-\phi(z)}\widehat{dv_{z}}dz,
 \end{eqnarray*}
one has
\begin{eqnarray} \label{2-9}
 \frac{d}{ds}|\Omega_s^{\ast}|_{n,\phi}=
 t'(s)\int_{S_{t(s)}}e^{-\phi(t(s))}\widehat{dv_{s}}.
\end{eqnarray}

We wish to point out the following fact:

\begin{lemma} \label{lemma-add}
For the function $t(s)$ in (\ref{2-9}), one has $t'(s)\neq0$.
\end{lemma}
\begin{proof}
Denote by $\mathcal{T}$ the set consisting of points, where the
function $f$ attains its critical values. By Sard's theorem (i.e.
the set of critical points of a smooth function has measure zero),
we can conclude that $\mathcal{T}$ has measure zero. Therefore, one
knows
\begin{eqnarray*}
\int_{\mathcal{T}}|\nabla f|^{-1}e^{-\phi|_{G(s)}}\widehat{dv_s}=0,
\end{eqnarray*}
and then
\begin{eqnarray} \label{add-222}
\int_{G(s)}|\nabla
f|^{-1}e^{-\phi|_{G(s)}}\widehat{dv_s}&=&\int_{G(s)\setminus\mathcal{T}}|\nabla
f|^{-1}e^{-\phi|_{G(s)}}\widehat{dv_s}
+\int_{\mathcal{T}}|\nabla f|^{-1}e^{-\phi|_{G(s)}}\widehat{dv_s}\nonumber\\
&=&\int_{G(s)\setminus\mathcal{T}}|\nabla
f|^{-1}e^{-\phi|_{G(s)}}\widehat{dv_s}.
\end{eqnarray}
This implies that there is no essential difference when doing
integrations over $G(s)\setminus\mathcal{T}$ or over $G(s)$. Based
on this reason, in the sequel, for convenience and simplicity, we
wish to integrate over $G(s)$ directly.

Therefore, combining (\ref{add-222}) with (\ref{2-8})-(\ref{2-9}),
one has
\begin{eqnarray*}
\int_{G(s)\setminus\mathcal{T}}|\nabla
f|^{-1}e^{-\phi|_{G(s)}}\widehat{dv_s}=t'(s)\int_{S_{t(s)}}e^{-\phi(t(s))}\widehat{dv_s},
\end{eqnarray*}
which implies $t'(s)\neq0$ since the LHS of the above equality
cannot be zero.
\end{proof}

Now let us go back to our discussion. Putting
(\ref{2-8})-(\ref{2-9}) into (\ref{2-5}) results in\footnote{~One
would see that similar conclusions can be obtained in the hemisphere
case and also the hyperbolic case.}
 \begin{eqnarray*}
\int_{G(s)}|\nabla
f|e^{-\phi|_{G(s)}}\widehat{dv_{s}}&\geq&\frac{\left(\int_{G^{\ast}(s)}e^{-\phi(t(s))}\widehat{dv_{s}}\right)^2}{\int_{G(s)}|\nabla
f|^{-1}e^{-\phi|_{G(s)}}\widehat{dv_{s}}}\nonumber\\
&=&\frac{\int_{S_{t(s)}}e^{-\phi(t(s))}\widehat{dv_{s}}}{-t'(s)}.
 \end{eqnarray*}
The above expression makes sense since $t'(s)\neq0$ by Lemma
\ref{lemma-add}. Therefore, by substituting the above inequality
into (\ref{2-3}), one can obtain
 \begin{eqnarray} \label{2-10}
\int_\Omega{|\nabla f|^2}d\eta&=&\int_0^{\sup
f}\int_{G(s)}|\nabla f|e^{-\phi|_{G(s)}}\widehat{dv_s}ds\nonumber\\
&\geq& -\int_0^{\sup
f}\frac{\int_{S_{t(s)}}e^{-\phi(t(s))}\widehat{dv_{s}}}{t'(s)}ds\nonumber\\
&=& -\int_0^{\sup
f}\left(\psi'(t(s))\right)^{2}t'(s)\int_{S_{t(s)}}e^{-\phi(t(s))}\widehat{dv_{s}}ds\nonumber\\
&=&
\int_{0}^{R}\left(\psi'(t)\right)^{2}\int_{S_{t(s)}}e^{-\phi(t(s))}\widehat{dv_{s}}dt\nonumber\\
&=& \int_{B_{R}(o)}|\nabla f^{\ast}|^{2}d\eta.
 \end{eqnarray}
The equality case in (\ref{2-10}) implies that
$\int_{G(0)}e^{-\phi|_{G(0)}} \widehat{dv} =
\int_{G^{\ast}(0)}e^{-\phi(t)} \widehat{dv}$ holds. So, one has
$G(0)\setminus E(0) = G^{\ast}(0)$, that is, $\Omega\setminus E(0) =
B_{R}(o)$. Moreover, this domain should lie entirely in the region
$B_{\mathcal{R}(h)}$ defined by (\ref{UIR-D}). Furthermore, by Lemma
\ref{lemma2-1}, we have
\begin{eqnarray*}
\lambda_{1,\phi}(\Omega)=\frac{\int_{\Omega}|\nabla
f|^2d\eta}{\int_{\Omega}f^2d\eta} \geq\frac{\int_{B_{R}(o)}|\nabla
f^{\ast}|^2
d\eta}{\int_{B_{R}(o)}(f^{\ast})^{2}d\eta}\geq\lambda_{1,\phi}(B_{R}(o)),
\end{eqnarray*}
which completes the proof of Theorem \ref{theo-1}.
\end{proof}

\begin{proof} [Proof of Theorem \ref{theo-4}]
Use an almost the same argument as that in the above proof of
Theorem \ref{theo-1} except replacing the usage of \textbf{Fact A}
and \textbf{Fact B} by the following fact:
\begin{itemize}
\item (\cite{BBMP}) Assume that the function $a:[0,+\infty)\rightarrow[0,+\infty)$ satisfies preconditions $a(t)$ is non-decreasing for $t\geq0$,
  $(a(z^{\frac{1}{n}})-a(0))z^{1-\frac{1}{n}}$ is convex, $z\geq0$,
  and moreover, assume that $\Omega\subset\mathbb{R}^n$ is a bounded open set with Lipschitz boundary $\partial\Omega$. Then
\begin{eqnarray*}
\int_{\partial\Omega}a(|x|)dx\geq\int_{\partial\Omega^*}a(|x|)dx,
\end{eqnarray*}
where $\partial\Omega^{\ast}$ is a sphere with center at the origin
and enclosing the weighted volume equal to that of $\Omega$.
\end{itemize}
Then the conclusion in Theorem \ref{theo-4} would follow naturally
by choosing $a(t)=e^{-\phi(t)}$.
\end{proof}

\subsection{The hemisphere case}

As we know, the Schwarz symmetrization can also be carried out on
hemispheres and hyperbolic spaces. For convenience, we will continue
to use notions and also the notations introduced at the beginning of
Subsection \ref{s2-1} to investigate the Faber-Krahn type
inequalities for the Witten-Laplcian in the hemisphere case and the
hyperbolic case.

\begin{lemma} \label{lemma2-2}
Assume that the function $\phi$ satisfies \textbf{Property 1} (with
$M^{n}$ chosen to be $\mathbb{S}^{n}_{+}$), where the point $o$
mentioned in \textbf{Property 1} should additionally be required to
be the base point of $\mathbb{S}^{n}_{+}$. Then we have
\begin{eqnarray*}
\int_{\Omega}f^{2}d\eta=\int_{B_{R}(o)}(f^{\ast})^2d\eta,
\end{eqnarray*}
 where $B_{R}(o)\subset\mathbb{S}^{n}_{+}$ is the geodesic ball defined as in Theorem \ref{theo-2}.
\end{lemma}

\begin{proof}
Formally, the computation for the assertion in Lemma \ref{lemma2-2}
is almost the same as that for (\ref{2-2}), and so we omit the
details here.
\end{proof}

We also need the following fact:

\begin{lemma} [\cite{BHW}]  \label{lemma2-3}
Let $\Omega\subset\mathbb{S}^{n}_{+}$ be a compact $n$-dimensional
domain with smooth boundary $\partial\Omega$. Let $H$ be the
normalized mean curvature of $\partial\Omega$. Let $V(x)=\cos
dist_{\mathbb{S}^n}(x,o)$. If $H$ is positive everywhere,
then\footnote{~In (\ref{2-11}), the Hausdorff measures of the domain
$\Omega$ and its boundary $\partial\Omega$ are given by $d\Omega$,
$dA$ respectively. This usage of notations does not match the
convention made at the beginning of Subsection \ref{s2-1}, and the
reason is that we wish to list here the original statement of the
conclusion in Lemma \ref{lemma2-3} proven firstly in the reference
\cite{BHW}.}
\begin{eqnarray} \label{2-11}
\int_{\partial\Omega}\frac{V}{H}dA\geq n\int_\Omega V d\Omega.
\end{eqnarray}
The equality in (\ref{2-11}) holds if and only if $\Omega$ is
isometric to a geodesic ball.
\end{lemma}

Now, we have:

\begin{proof} [Proof of Theorem \ref{theo-2}] Applying the
co-area formula, we have
\begin{eqnarray} \label{2-12}
\int_\Omega{|\nabla f|^2}\cos tdv=\int_0^{\sup f}\int_{G(s)}|\nabla
f|\cos(t|_{G(s)})\widehat{dv_s}ds.
\end{eqnarray}
We can obtain by using the Cauchy-Schwarz inequality that
\begin{eqnarray}  \label{2-13}
\int_{G(s)}|\nabla f|\cos(t|_{G(s)})\widehat{dv_{s}}\geq
\frac{\left(\int_{G(s)}\cos(t|_{G(s)})\widehat{dv_{s}}\right)^2}{\int_{G(s)}|\nabla
f|^{-1}\cos(t|_{G(s)})\widehat{dv_{s}}}.
\end{eqnarray}
By Lemma \ref{lemma2-3} and the assumption that $H$ is a positive
constant, one has $\int_{G(s)}\cos(t|_{G(s)})\widehat{dv_{s}}\geq
\int_{G^{\ast}(s)}\cos t(s) \widehat{dv_{s}}$, and then (\ref{2-13})
becomes
\begin{eqnarray}  \label{2-14}
\int_{G(s)}|\nabla f|\cos(t|_{G(s)})\widehat{dv_{s}}\geq
\frac{\left(\int_{G^{\ast}(s)}\cos t(s)
\widehat{dv_{s}}\right)^2}{\int_{G(s)}|\nabla
f|^{-1}\cos(t|_{G(s)})\widehat{dv_{s}}}.
\end{eqnarray}
On the other hand, one has
\begin{eqnarray*} \label{2-15}
\int_{G^{\ast}(s)}|\nabla f^{\ast}|\cos t(s)\widehat{dv_{s}}=
\frac{\left(\int_{G^{\ast}(s)}\cos
t(s)\widehat{dv_{s}}\right)^2}{\int_{G^{\ast}(s)}|\nabla
f^{\ast}|^{-1}\cos t(s)\widehat{dv_{s}}},
\end{eqnarray*}
 since $|\nabla f^{\ast}|$ and $\cos t(s)$ are constant on the sphere
 $G^{\ast}(s)$. Notice that
 \begin{eqnarray*}
|\Omega_r|_{n,\phi}=\int_{\Omega_r}\cos t(r) dv=\int_r^{\sup
f}\int_{G(s)}|\nabla f|^{-1}\cos(t|_{G(s)})\widehat{dv_{s}}ds,
 \end{eqnarray*}
and so it follows that
 \begin{eqnarray*} \label{2-16}
 \left(|\Omega_r|_{n,\phi}\right)'(s)=-\int_{G(s)}|\nabla
 f|^{-1}\cos(t|_{G(s)})\widehat{dv_{s}},
 \end{eqnarray*}
which implies
 \begin{eqnarray} \label{2-17}
-\int_{G(s)}|\nabla
 f|^{-1}\cos(t|_{G(s)})\widehat{dv_{s}}=\frac{d}{ds}|\Omega_s|_{n,\phi}=\frac{d}{ds}|\Omega^{\ast}_s|_{n,\phi}.
 \end{eqnarray}
 Since
 \begin{eqnarray*}
|\Omega_s^{\ast}|_{n,\phi}=\int_0^{t(s)}\int_{\partial B_{z}(o)}\cos
z\widehat{dv_{z}}dz,
 \end{eqnarray*}
one has
\begin{eqnarray} \label{2-18}
 \frac{d}{ds}|\Omega_s^{\ast}|_{n,\phi}=
 t'(s)\int_{S_{t(s)}}\cos t(s)\widehat{dv_{s}}.
\end{eqnarray}
Putting (\ref{2-17})-(\ref{2-18}) into (\ref{2-14}) results in
 \begin{eqnarray*}
\int_{G(s)}|\nabla
f|\cos(t|_{G(s)})\widehat{dv_{s}}&\geq&\frac{\left(\int_{G^{\ast}(s)}\cos
t(s) \widehat{dv_{s}}\right)^2}{\int_{G(s)}|\nabla
f|^{-1}\cos(t|_{G(s)})\widehat{dv_{s}}}\nonumber\\
&=&\frac{\int_{S_{t(s)}}\cos t(s)\widehat{dv_{s}}}{-t'(s)}.
 \end{eqnarray*}
Therefore, by substituting the above inequality into (\ref{2-12}),
one has
\begin{eqnarray*}
\int_\Omega{|\nabla f|^2}d\eta&=&\int_0^{\sup
f}\int_{G(s)}|\nabla f|\cos(t|_{G(s)})\widehat{dv_s}ds\nonumber\\
&\geq& -\int_0^{\sup
f}\frac{\int_{S_{t(s)}}\cos t(s)\widehat{dv_{s}}}{t'(s)}ds\nonumber\\
&=& -\int_0^{\sup
f}\left(\psi'(t(s))\right)^{2}t'(s)\int_{S_{t(s)}}\cos t(s)\widehat{dv_{s}}ds\nonumber\\
&=&
\int_{0}^{R}\left(\psi'(t)\right)^{2}\int_{S_{t(s)}}\cos t(s)\widehat{dv_{s}}dt\nonumber\\
&=& \int_{B_{R}(o)}|\nabla f^{\ast}|^{2}d\eta.
 \end{eqnarray*}
 Together with Lemma \ref{lemma2-2}, it follows that
  \begin{eqnarray} \label{2-19}
 \lambda_{1,\phi}(\Omega)=\frac{\int_{\Omega}|\nabla f|^2d\eta}{\int_{\Omega}f^2d\eta}
\geq\frac{\int_{B_{R}(o)}|\nabla f^{\ast}|^2
d\eta}{\int_{B_{R}(o)}(f^{\ast})^2d\eta}\geq\lambda_{1,\phi}(B_{R}(o)).
  \end{eqnarray}
Especially, if the equality in (\ref{2-19}) was achieved, then the
equality in (\ref{2-13}) and (\ref{2-14}) can be attained
simultaneously, and the rigidity assertion in Theorem \ref{theo-2}
follows by using Lemma \ref{lemma2-3} directly. This completes the
proof of Theorem \ref{theo-2}.
\end{proof}

\subsection{The hyperbolic case}

\begin{proof} [Proof of Theorem \ref{theo-3}]
It is not hard to see that similar to Lemma \ref{lemma2-1}, in the
hyperbolic case one also has the $L^{2}$ integral (w.r.t. the
weighted density $d\eta$) unchanged after the Schwarz symmetrization
under the constraint of fixed weighted volume. Besides, if one looks
at the proofs of Theorems \ref{theo-1} and \ref{theo-2}, one would
find that in the two different cases (i.e. the case of Euclidean
spaces and the case of hemispheres), the co-area formula and most
subsequent calculations look similarly in form. The key difference
for those two cases is the usage of weighted isoperimetric
inequalities (i.e. the way of dealing with (\ref{2-5}) and
(\ref{2-13})) \emph{properly}. Based on these facts, then  using an
almost the same argument as that in the proof of Theorem
\ref{theo-1}, together with the help of \textbf{Fact C} (i.e., the
geometric isoperimetric inequality in $\mathbb{H}^n$ under the
constraint of fixed weighted volume), we can get the spectral
isoperimetric inequality and the rigidity in Theorem \ref{theo-3}.
\end{proof}

\section{The Hong-Krahn-Szeg\H{o} type inequalities for the
Witten-Laplcian} \label{s3}
\renewcommand{\thesection}{\arabic{section}}
\renewcommand{\theequation}{\thesection.\arabic{equation}}
\setcounter{equation}{0}

For the Dirichlet eigenvalue problem (\ref{eigen-D11}), we know from
Section \ref{s1} that its admissible space is the Sobolev space
$W^{1,2}_{0,\phi}(\Omega)$. Using the inner product (\ref{inn-p}),
one can define the $L^{2}$ space $\widehat{L}^{2}(\Omega)$ w.r.t.
the weighted density as follows: we say that
$u\in\widehat{L}^{2}(\Omega)$ if
 \begin{eqnarray*}
 \int_{\Omega}u^{2}e^{-\phi}dv<\infty.
 \end{eqnarray*}
Before giving the proof of the Hong-Krahn-Szeg\H{o} type
inequalities for the second Dirichlet eigenvalue of the
Witten-Laplcian, we need the following facts.

\begin{lemma} (Nodal domain theorem for the Witten-Laplacian, \cite{CMW})
\label{lemma3-1} For the Dirichlet eigenvalue problem
(\ref{eigen-D11}), its eigenvalues consist of a non-decreasing
sequence (\ref{sequence-3}). Denote by $f_{i}$ an eigenfunction of
the $i$-th eigenvalue $\lambda_{i,\phi}$, $i=1,2,3,\cdots$, and
$\{f_{1},f_{2},f_{3},\cdots\}$ forms a complete orthogonal basis of
${\widehat{L}^2(\Omega)}$. Then for each $k=1,2,3,\cdots$, the
number of nodal domains of $f_{k}$ is less than or equal to $k$.
\end{lemma}

\begin{remark}
\rm{ (1) By Lemma \ref{lemma3-1}, one easily knows that the
eigenfunction $f_{1}$ does not change sign on $\Omega$, and
$\lambda_{1,\phi}$ has multiplicity $1$. Without loss of generality,
we can assume $f_{1}>0$ on $\Omega$. Besides, in $\Omega$, the
complement of the nodal set of
 eigenfunction $f_{2}$ of the second Dirichlet eigenvalue
$\lambda_{2,\phi}$ has precisely two components. That is to say,
$f_{2}$ has two nodal domains. \\
 (2) BTW, we have pointed out in Remark 1.3 of our another work \cite{CMW} that maybe spectral geometers have already
 known the conclusion of Lemma \ref{lemma3-1}, and we still formally write it down
 therein for the completion of the
structure of \cite{CMW}. In fact, by making necessary changes to the
proof of Courant-type theorem for the characterization of nodal
domains to eigenfunctions of the Laplacian in the Riemannian case
given by B\'{e}rard-Meyer \cite{BM}, one might get our proof for the
conclusion of Lemma \ref{lemma3-1} shown in \cite{CMW}.
 }
\end{remark}

\begin{lemma} (\cite{DM}) \label{lemma3-3}
Domain monotonicity of eigenvalues with vanishing Dirichlet data
also holds for the Dirichlet eigenvalues of the weighted Laplacian.
\end{lemma}

Now, we have:

\begin{proof} [A proof of Theorem \ref{theo-5} or \ref{theo-6}]
By Lemma \ref{lemma3-1}, one knows that the eigenfunction $f_{2}$
has two nodal domains and its nodal set lies inside $\Omega$. Denote
by $\Gamma$ the nodal set of $f_{2}$. $\Gamma$ divides the domain
$\Omega$ into two parts $D_1$ and $D_2$. Without loss of generality,
assume that $f_{2}|_{D_1}>0$ and $f_{2}|_{D_2}<0$. Then it is easy
to see that
\begin{eqnarray} \label{3-1}
\left\{
\begin{array}{ll}
\Delta_{\phi}f_{2}+\lambda_{2,\phi}(\Omega)f_{2}=0\qquad & \mathrm{in}~D_1, \\[0.5mm]
f_{2}=0 \qquad & \mathrm{on}~\partial D_1,
\end{array}
\right.
\end{eqnarray}
and
\begin{eqnarray} \label{3-2}
\left\{
\begin{array}{ll}
\Delta_{\phi}f_{2}+\lambda_{2,\phi}(\Omega)f_{2}=0\qquad & \mathrm{in}~D_2, \\[0.5mm]
f_{2}=0 \qquad & \mathrm{on}~\partial D_2,
\end{array}
\right.
\end{eqnarray}
In fact, the nodal set $\Gamma$ also divides the boundary
$\partial\Omega$ into two parts, and let us call them
$\mathcal{C}_{1}$ and $\mathcal{C}_{2}$. It is not hard to see that
$\mathcal{C}_{1}$ and $\Gamma$ surround one of $D_1$ and $D_2$, and
without loss of generality, let us say $D_1$. This implies that the
boundary $\partial D_1$ of $D_1$ satisfies $\partial
D_1=\mathcal{C}_{1}\cup\Gamma$. Correspondingly, one has $\partial
D_2=\mathcal{C}_{2}\cup\Gamma$. From (\ref{3-1}) and (\ref{3-2}),
one knows that $f_{2}$ satisfies the eigenvalue problem
(\ref{eigen-D11}) with $\Omega=D_1$ or $\Omega=D_2$, and moreover,
$f_2$ does not change sign on $D_{i}$, $i=1,2$. Hence, we have
$\lambda_{1,\phi}(D_1)=\lambda_{2,\phi}(\Omega)=\lambda_{1,\phi}(D_2)$,
and $f_{2}$ can be treated as an eigenfunction of
$\lambda_{1,\phi}(D_{i})$, $i=1,2$. Denote by $B_{R_i}(o)$ the
(geodesic) ball in $\mathbb{R}^n$ (or $\mathbb{H}^n$) centered at
the origin $o$ and radius $R_i$ such that its weighted volume equals
that of $D_{i}$, $i=1,2$, that is,
$|B_{R_i}(o)|_{n,\phi}=|D_{i}|_{n,\phi}$. Then by Theorem
\ref{theo-1} (or Theorem \ref{theo-3}), we know that
\begin{eqnarray*}
\lambda_{2,\phi}(\Omega)\geq\lambda_{1,\phi}(B_{R_1}(o)),\qquad
\lambda_{2,\phi}(\Omega)\geq\lambda_{1,\phi}(B_{R_2}(o))
\end{eqnarray*}
hold simultaneously. Hence, one has
\begin{eqnarray*}
\lambda_{2,\phi}(\Omega)\geq\max\{\lambda_{1,\phi}(B_{R_1}(o)),\lambda_{1,\phi}(B_{R_2}(o))\}.
\end{eqnarray*}
We may suppose that $|D_{1}|_{n,\phi}\leq|D_{2}|_{n,\phi}$. So,
$R_{1}\leq R_{2}$, and by Lemma \ref{lemma3-3} we have
$\lambda_{1,\phi}(B_{R_1}(o))\geq\lambda_{1,\phi}(B_{R_2}(o))$.
Therefore, in this setting, finding the greatest lower bound for the
second eigenvalue $\lambda_{2}(\Omega)$ among domains with the fixed
weighted volume $|\Omega|_{n,\phi}=const.$, it is sufficient to
minimize $\lambda_{1,\phi}(B_{R_1}(o))$. Since
$|D_{1}|_{n,\phi}\leq|D_{2}|_{n,\phi}$ and
$|D_{1}|_{n,\phi}+|D_{2}|_{n,\phi}=|\Omega|_{n,\phi}$, the maximal
possibility for the weighted volume of $D_1$ is that
$|D_{1}|_{\phi}=|\Omega|_{n,\phi}/2$. Hence, there exists
$\widetilde{R}>0$ such that
$|B_{\widetilde{R}}(o)|_{\phi}=|\Omega|_{n,\phi}/2$, and by Lemma
\ref{lemma3-3}, in this situation, the eigenvalue
$\lambda_{1,\phi}(B_{\widetilde{R}}(o))$ minimizes the eigenvalue
functional $\lambda_{1,\phi}(B_{R_1}(o))$ as $R_{1}$ changes. Hence,
one has
$\lambda_{2,\phi}(\Omega)\geq\lambda_{1,\phi}(B_{\widetilde{R}}(o))$,
and the eigenvalue $\lambda_{1,\phi}(B_{\widetilde{R}}(o))$ equals
the minimum value of the eigenvalue functional
$\lambda_{2,\phi}(\Omega)$ under the constraint of weighted volume
$|\Omega|_{n,\phi}=const.$ fixed. This completes the proof.
\end{proof}

\section{The Szeg\H{o}-Weinberger type inequalities for the
Witten-Laplcian} \label{s4}
\renewcommand{\thesection}{\arabic{section}}
\renewcommand{\theequation}{\thesection.\arabic{equation}}
\setcounter{equation}{0}

This section devotes to giving isoperimetric inequalities for the
first nonzero Neumann eigenvalue of the Witten-Laplacian under the
constraint of weighted volume fixed. Before that, we need the
following fact.

\begin{theorem}  \label{theo4-1}
Assume that $B_{R}(o)$ is a geodesic ball of radius $R$ and centered
at some point $o$ in the $n$-dimensional complete simply connected
Riemannian manifold $\mathbb{M}^{n}({\kappa})$ with constant
sectional curvature $\kappa\in\{-1,0,1\}$, and that  $\phi$ is a
radial function w.r.t. the distance parameter $t:=d(o,\cdot)$, which
is also a non-increasing convex function. Then the eigenfunctions of
the first nonzero Neumann eigenvalue $\mu_{1,\phi}(B_{R}(o))$ of the
Witten-Laplacian
 on $B_{R}(o)$ should have the form $T(t)\frac{x_i}{t}$,
 $i=1,2,\cdots,n$, where $T(t)$
satisfies
\begin{eqnarray}  \label{4-1}
\left\{
\begin{array}{ll}
T''+\left(\frac{(n-1)C_{\kappa}}{S_{\kappa}}-\phi'\right)T'+\left(\mu_{1,\phi}(B_{R}(o))-(n-1)S_{\kappa}^{-2}\right)T=0,\\[1mm]
T(0)=0,~T'(R)=0,~T'|_{[0,R)}\neq0.
\end{array}
\right.
\end{eqnarray}
 Here $C_{\kappa}(t)=\left(S_{\kappa}(t)\right)'$ and
 \begin{eqnarray*}
S_{\kappa}(t)=\left\{
\begin{array}{lll}
\sin t,~~&\mathrm{if}~\mathbb{M}^{n}(\kappa)=\mathbb{S}^{n}_{+}, \\
t,~~&\mathrm{if}~\mathbb{M}^{n}(\kappa)=\mathbb{R}^{n},\\
\sinh t,~~&\mathrm{if}~\mathbb{M}^{n}(\kappa)=\mathbb{H}^{n},
\end{array}
\right.
 \end{eqnarray*}
  with $\mathbb{S}^{n}_{+}$ the $n$-dimensional hemisphere of radius $1$.
\end{theorem}

The proof of the above fact is a little bit long, and looks like it
does not have close relation with the main content of this section.
Hence, we wish to leave the proof in Appendix -- Section \ref{s5}.

\begin{remark}
\rm{it is not hard to see in Section \ref{s5} that $x_{i}$,
$i=1,2,\cdots,n$, are coordinate functions of the globally defined
orthonormal coordinate system set up in $\mathbb{M}^{n}(\kappa)$. }
\end{remark}

We construct an auxiliary function $h(t)$ such that
\begin{eqnarray} \label{4-2}
h(t)=\left\{
\begin{array}{ll}
T(t),~~&0\leq t\leq R, \\[1mm]
 T(R),~~&t>R.
\end{array}
\right.
 \end{eqnarray}

\begin{lemma} \label{lemma4-3}
Assume that the function $\phi$ satisfies \textbf{Property 1} (with
$M^{n}$ chosen to be $\mathbb{R}^n$ and additionally the point $o$
required to be in the convex hull of $\Omega$, i.e.
$o\in{\mathrm{hull}}(\Omega)$). Assume that $T(t)$ is monotonically
non-decreasing determined by the system (\ref{4-1}). Then $h(t)$ is
monotonically non-decreasing, and $(h')^2+(n-1)h^2/t^2$ is
monotonically non-increasing.
\end{lemma}

\begin{proof}
First, it is easy to check that $h(t)$ defined by (\ref{4-2}) is
non-decreasing. Besides, by a direct calculation, one has
 \begin{eqnarray*}
\frac{d}{dt}\left[(h')^2+\frac{(n-1)h^2}{t^2}\right]=2h'h''+2(n-1)\frac{thh'-h^2}{t^3}.
 \end{eqnarray*}
 Together with (\ref{4-1}), we have
  \begin{eqnarray*}
\frac{d}{dt}\left[(h')^2+\frac{(n-1)h^2}{t^2}\right]=-2\mu_{1,\phi}(B_{R}(o))hh'-(n-1)\frac{(th'-h)^2}{t^3}+2(h')^2\phi'\leq0,
  \end{eqnarray*}
  which implies the second assertion of the lemma directly.
\end{proof}

\begin{lemma} \label{lemma4-4}
Assume that $\Omega$ is a bounded domain in $\mathbb{R}^n$ (or
$\mathbb{H}^n$) with smooth boundary. If
$|\Omega|_{n,\phi}=|B_{R}(o)|_{n,\phi}$, with $B_{R}(o)$ be the
(geodesic) ball defined as in Theorem \ref{theo-7} (or Theorem
\ref{theo-8}), and the non-constant functions $u(t)$ and $v(t)$
defined on $[0,+\infty)$ are monotonically non-increasing and
non-decreasing respectively, then
\begin{eqnarray*}
&&\int_\Omega v(|x|)d\eta\geq\int_{B_{R}(o)}v(|x|)d\eta,\\
&&\int_\Omega u(|x|)d\eta\leq\int_{B_{R}(o)}u(|x|)d\eta.
\end{eqnarray*}
The equality holds if and only if $\Omega=B_{R}(o)$ (or $\Omega$ is
isometric to $B_{R}(o)$).
\end{lemma}

\begin{proof}
Assume that $Q=\Omega\cap B_{R}(o)$, and then we have
 \begin{eqnarray*}
\int_\Omega v(|x|)d\eta&=&\int_Q v(|x|)d\eta+\int_{\Omega\setminus Q} v(|x|)d\eta\\
&\geq&\int_{Q} v(|x|)d\eta+v(R)\int_{\Omega\setminus Q}d\eta.
 \end{eqnarray*}
Similarly, one has
 \begin{eqnarray*}
\int_{B_{R}(o)} v(|x|)d\eta&=&\int_Q v(|x|)d\eta+\int_{B_{R}(o)\setminus Q} v(|x|)d\eta\\
&\leq&\int_{Q} v(|x|)d\eta+v(R)\int_{B_{R}(o)\setminus Q}d\eta.
\end{eqnarray*}
Since $|\Omega|_{n,\phi}=|B_{R}(o)|_{n,\phi}$, then
$\int_{\Omega\setminus Q}d\eta=\int_{B_{R}(o)\setminus Q}d\eta$, and
substituting this fact into the above two inequalities yields
 \begin{eqnarray*}
\int_\Omega v(|x|)d\eta\geq\int_{B_{R}(o)} v(|x|)d\eta.
 \end{eqnarray*}
 Specially, when the equality holds, one has
  \begin{eqnarray*}
  \int_{\Omega\setminus Q}
 v(|x|)d\eta=v(R)\int_{\Omega\setminus Q}d\eta, \qquad \int_{B_{R}(o)\setminus Q} v(|x|)d\eta=v(R)\int_{B_{R}(o)\setminus
 Q}d\eta
  \end{eqnarray*}
   simultaneously. Since the non-constant function $v$ is
   non-increasing, $\Omega$ is the ball $B_{R}(o)$ (or $\Omega$ is
isometric to $B_{R}(o)$). The situation for the non-constant
function $u$ can be dealt with similarly.
\end{proof}

Now, we have:

\begin{proof} [Proof of Theorem \ref{theo-7}]
Define $f(t):=h(t)\frac{x_i}{t}$, where $i$ is chosen to be an
integer of the set $\{1,2,\cdots,n\}$. Then applying the Brouwer's
fixed point theorem and choosing a suitable coordinate origin
$o\in{\mathrm{hull}}(\Omega)$, we can assure
$\int_{\Omega}fd\eta=0$. This can be done by using a very similar
argument to that on \cite[pp. 634-635]{HFW}. In fact, one can also
check our another work \cite{CM1} where we have given a detailed
explanation on how to get the suitable coordinate system such that
$\int_{\Omega}fd\eta=0$. By the characterization (\ref{chr-4}), and
by using a similar calculation to (2.9)-(2.10) on \cite[page
635]{HFW}, one has
  \begin{eqnarray*}
\mu_{1,\phi}(\Omega)\leq\frac{\int_{\Omega}
\left[(h')^2+\frac{(n-1)h^2}{t^2}\right]d\eta}{\int_\Omega h^2
d\eta}.
\end{eqnarray*}
On the other hand, by Lemma \ref{lemma4-3} and Lemma \ref{lemma4-4},
we have
\begin{eqnarray*}
\int_{\Omega}
\left[(h')^2+\frac{(n-1)h^2}{t^2}\right]d\eta\leq\int_{B_{R}(o)}
\left[(h')^2+\frac{(n-1)h^2}{t^2}\right]d\eta
\end{eqnarray*}
and
 \begin{eqnarray*}
\int_\Omega h^2d\eta\geq\int_{B_{R}(o)} h^2 d\eta.
 \end{eqnarray*}
Therefore, we have
 \begin{eqnarray*}
\mu_{1,\phi}(\Omega)&\leq&\frac{\int_{\Omega}
\left[(h')^2+\frac{(n-1)h^2}{t^2}\right]d\eta}{\int_\Omega h^2
d\eta}\\
&\leq&\frac{\int_{B_{R}(o)}
\left[(h')^2+\frac{(n-1)h^2}{t^2}\right]d\eta}{\int_{B_{R}(o)} h^2
d\eta}\\
&=&\mu_{1,\phi}(B_{R}(o)),
 \end{eqnarray*}
  which together with the description of the equality case in Lemma
  \ref{lemma4-4} implies the assertion of Theorem \ref{theo-7}
  directly.
\end{proof}

\begin{proof} [Proof of Theorem \ref{theo-8}]
We still use $f(t)$ as the trail function, but now the distance
should be the Riemannian distance in the hyperbolic space
$\mathbb{H}^{n}$. In the hyperbolic case, using a similar argument
to that in the proof of Theorem \ref{theo-7}, we have
 \begin{eqnarray} \label{4-3}
\mu_{1,\phi}(\Omega)\leq\frac{\int_{\Omega}
\left[(h')^2+\frac{(n-1)h^2}{(\sinh t)^2}\right]d\eta}{\int_\Omega
h^2 d\eta}.
 \end{eqnarray}
On the other hand,
\begin{eqnarray*}
\frac{d}{dt}\left[(h')^2+\frac{(n-1)h^2}{(\sinh
t)^2}\right]=2h'h''+2(n-1)\frac{hh'\sinh t-h^{2}\cosh t}{(\sinh
t)^{3}}.
\end{eqnarray*}
Putting (\ref{4-1}) into the above equality and using the facts
$\sinh t\geq0$, $\cosh t\geq1$ for $t\geq0$, one has
\begin{eqnarray*}
&&\frac{d}{dt}\left[(h')^2+\frac{(n-1)h^2}{(\sinh
t)^2}\right]\\
&&\qquad=-2\mu_{1,\phi}(B_{R}(o))hh'+2(h')^2\phi'-\frac{2(n-1)\cosh t}{\sinh t}(h')^2\\
&&\qquad\qquad -\frac{2(n-1)\cosh t}{\sinh^3 t}h^2+\frac{4(n-1)}{\sinh^2 t}hh'\\
&&\qquad\leq-2\mu_{1,\phi}(B_{R}(o))hh'+2(h')^2\phi'-\frac{2(n-1)}{\sinh t}(h')^2\\
&&\qquad\qquad -\frac{2(n-1)}{\sinh^3 t}h^2+\frac{4(n-1)}{\sinh^2 t}hh'\\
&&\qquad=-2\mu_{1,\phi}(B_{R}(o))hh'+2(h')^2\phi'-2(n-1)\frac{(h')^2\sinh^2 t+h^2-2hh'\sinh t}{\sinh^3 t}\\
&&\qquad=-2\mu_{1,\phi}(B_{R}(o))hh'+2(h')^2\phi'-2(n-1)\frac{(h'\sinh
t-h)^2}{\sinh^{3} t}\\
&&\qquad\leq 0.
\end{eqnarray*}
Then, by applying Lemma \ref{lemma4-4}, we have
\begin{eqnarray*}
\int_{\Omega} \left[(h')^2+\frac{(n-1)h^2}{\sinh ^2
t}\right]d\eta\leq\int_{B_{R}(o)} \left[(h')^2+\frac{(n-1)h^2}{\sinh
^2 t}\right]d\eta
\end{eqnarray*}
and
 \begin{eqnarray*}
\int_\Omega h^2d\eta\geq\int_{B_{R}(o)} h^2 d\eta.
 \end{eqnarray*}
Therefore, from (\ref{4-3}) we can obtain
 \begin{eqnarray*}
\mu_{1,\phi}(\Omega)\leq\frac{\int_{\Omega}
\left[(h')^2+\frac{(n-1)h^2}{\sinh ^2 t}\right]d\eta}{\int_\Omega
h^2 d\eta}\leq\frac{\int_{B_{R}(o)}
\left[(h')^2+\frac{(n-1)h^2}{\sinh ^2
t}\right]d\eta}{\int_{B_{R}(o)} h^2 d\eta}=\mu_{1,\phi}(B_{R}(o)),
 \end{eqnarray*}
  which together with the description of the equality case in Lemma
  \ref{lemma4-4} implies the assertion of Theorem \ref{theo-8}
  directly.
\end{proof}

\section{Appendix} \label{s5}
\renewcommand{\thesection}{\arabic{section}}
\renewcommand{\theequation}{\thesection.\arabic{equation}}
\setcounter{equation}{0}

Now, in this section we give a proof of Theorem \ref{theo4-1} in
details. Assume that $f$ is an eigenfunction of the Witten-Laplace
operator $\Delta_{\phi}$, and $f$ can be decomposed into
$T(t)G(\xi)$, where $t:=d(o,\cdot)$ stands for the Riemannian
distance to the point $o$, and $\xi\in S_{o}^{n-1}\subset
T_{o}\mathbb{M}^{n}({\kappa})$. A simple calculation gives us that
 \begin{eqnarray*}
0=\Delta_\phi f+\mu
f=S_{\kappa}^{1-n}(S_{\kappa}^{n-1}T')'G-S_{\kappa}^2T v_l
G-\phi'T'G+\mu T G,
 \end{eqnarray*}
 where $v_l$ denotes the closed eigenvalue of the Laplacian on the unit $(n-1)$-sphere $\mathbb{S}^{n-1}$, i.e., $v_l=l(l+n-2)$,
 $l=0,1,2,\cdots$. Simplifying the above equation gives us a
 second-order ODE as follows
 \begin{eqnarray} \label{5-1}
T''+\left[\frac{(n-1)C_{\kappa}}{S_{\kappa}}-\phi'\right]T'+\left(\mu-\frac{v_l}{S_{\kappa}^2}\right)T=0,
 \end{eqnarray}
where $C_{\kappa}(t)=S_{\kappa}'(t)$. For the Neumann eigenvalue
problem of the Witten-Laplacian $\Delta_{\phi}$, in order to ensure
the smoothness of the function $T$, we have:
 \begin{itemize}
\item When $l=0$, $T'(0)=0$;
\item $T(t)\thicksim t^{l}$, $l=1,2,\cdots$;
\item $T$ satisfies the Neumann boundary condition $T'(R)=0$.
 \end{itemize}
Choosing a relatively small positive number $\epsilon$ and letting
$p(t)=e^{\int_\epsilon^{t}(\frac{(n-1)C_{\kappa}}{S_{\kappa}}-\phi')
ds}$, we can simplify (\ref{5-1}) into a Sturm-Liouville equation
\begin{eqnarray} \label{5-2}
(pT')'+(\mu-v_l S_{\kappa}^{-2})pT=0.
\end{eqnarray}
Assume that for a fixed $v_{l}$, $\mu_{l,j,\phi}$, $j=1,2,\cdots$,
is the $j$-th eigenvalue related to $v_l$, and $T_{l,j,\phi}$
denotes an eigenfunction belonging to $\mu_{l,j,\phi}$. Here the
purpose that we put the symbol $\phi$ in the subscript of
$\mu_{l,j,\phi}$ is to emphasize that theoretically
$\mu_{l,j,\phi}$, $T_{l,j,\phi}$ have close relation with the
function $\phi$ since the function $p(t)$ in the equation
(\ref{5-2}) depends on $\phi'(t)$. In this setting, the equation
(\ref{5-2}) can be rewritten as
\begin{eqnarray} \label{5-3}
(pT_{l,j,\phi}')'+\left(\mu_{l,j,\phi}-v_l
S_{\kappa}^{-2}\right)pT_{l,j,\phi}=0,
\end{eqnarray}
which implies
 \begin{eqnarray}  \label{add-extra}
\int_0^R T_{l,j,\phi} ~T_{l,k,\phi} p dt=0, \qquad \mathrm{when}
~~\mu_{l,j,\phi}\neq\mu_{l,k,\phi}.
 \end{eqnarray}
Moreover, one can normalize $T$ such that
\begin{eqnarray*}
\int_0^{R} T_{l,j,\phi} T_{l,j,\phi} p dt=1.
\end{eqnarray*}

For an equation of the form similar to (\ref{5-3}), we have the
following fact.
\begin{lemma} \label{lemma5-1}
Assume that functions $f$ and $g$ satisfy separately the equations
\begin{eqnarray}
&& (pf')'+(\alpha-\sigma(t))pf=0, \label{5-4}\\
&& (pg')'+(\beta-\tau(t))pg=0, \label{5-5}
\end{eqnarray}
 and also the boundary conditions given as in the system
 (\ref{4-1}). Then we have
  \begin{eqnarray*}
p(fg'-f'g)(t)=\int_0^t \left[\alpha-\beta+(\tau-\sigma)\right]pfgds.
  \end{eqnarray*}
\end{lemma}

\begin{proof}
Multiplying both sides of the equation (\ref{5-4}) by $g$,
multiplying both sides of the equation (\ref{5-5}) by $f$, and then
making difference yields
\begin{eqnarray*}
(pf')'g-(pg')'f+\left[\alpha-\beta+(\tau(t)-\sigma(t))\right]pfg=0.
\end{eqnarray*}
Integrating both sides of the above equality from $0$ to $t$, and
  using the boundary conditions given as in the system
 (\ref{4-1}), one can get the assertion of Lemma \ref{5-1}
 directly.
\end{proof}

By the standard Sturm-Liouville theory for second-order ODEs, we
know that $T_{l,j,\phi}$ has exactly $j-1$ zeros on the interval
$(0, R)$. So, $T_{l,1,\phi}$ keeps its sign unchanged on $(0, R)$.
Without loss of generality, we may assume that $T_{l,1,\phi}$ and
$T_{k,1,\phi}$ are both greater than $0$, where $l<k$. Then, by
Lemma \ref{lemma5-1}, when $t=R$, we have
$\mu_{l,1,\phi}(R)<\mu_{k,1,\phi}(R)$, $l<k$. Since for the
eigenvalue problem (\ref{eigen-N11}), we know from its sequence
(\ref{sequence-4}) that $\mu_{1,\phi}=\mu_{0,1,\phi}=0$. Hence, if
one wants to get the first non-zero Neumann eigenvalue
$\mu_{1,\phi}$ of the Witten-Laplacian on $B_{R}(o)$, one only needs
to know exactly which one is smaller between $\mu_{0,2,\phi}$ and
$\mu_{1,1,\phi}$.

The following lemma is important and fundamental.
\begin{lemma}
When $l\geq 1$, $T_{l,j,\phi}'$ has only $j-1$ zeros in the interval
$(0,R)$.
\end{lemma}

\begin{proof}
From (\ref{5-3}), one has
\begin{eqnarray} \label{5-6}
pT_{l,j,\phi}''+p'T_{l,j,\phi}'+\mu_{l,j,\phi}pT_{l,j,\phi}-v_l
S_{\kappa}^{-2}pT_{l,j,\phi}=0.
\end{eqnarray}
Since $T_{l,1,\phi}$ has no zero points on the interval $(0, R)$, we
can assume that $T_{l,1,\phi}$ is greater than $0$. According to the
boundary conditions, if $T_{l,1,\phi}'$ is not constantly greater
than $0$ on the interval $(0, R)$, then there exists a $t_0<t_1$
such that $T_{l,1,\phi}''(t_0)\leq0$, $T_{l,1,\phi}'(t_0)=0$ and
$T_{l,1,\phi}''(t_1)\geq0$, $T_{l,1,\phi}'(t_1)=0$ hold true.
Together with (\ref{5-6}), we can obtain
\begin{eqnarray*}
S_{\kappa}^2(t_0)\geq\frac{v_{l}}{\mu_{l,1,\phi}}\geq
S_{\kappa}^2(t_1).
\end{eqnarray*}
Due to the increasing property of $S_{\kappa}(t)$, this contradicts
with $t_0<t_1$. So, $T'_{l,1,\phi}$ has no zero points in the
interval $(0,R)$. For the case  $T'_{l,j,\phi}$, $j>1$,  one only
needs to repeat the above argument in each nodal domain.
\end{proof}

It is not hard to know that the function $T_{0,2}$ satisfies
\begin{eqnarray}  \label{5-7}
\left\{
\begin{array}{ll}
(pT_{0,2,\phi}')'+\mu_{0,2,\phi}pT_{0,2,\phi}=0, \\[2mm]
 T_{0,2,\phi}'(0)=T_{0,2,\phi}'(R)=0.
\end{array}
\right.
 \end{eqnarray}
Since $T_{0,1,\phi}$ is a non-zero constant function, and
$T_{0,2,\phi}$ is orthogonal to $T_{0,1,\phi}$ in the sense of
(\ref{add-extra}), we know that $T_{0,2,\phi}$ changes sign on the
interval $(0, R)$. Therefore, we may assume that $T_{0,2,\phi}$ is
positive on some interval $(0, r_0)$ and $T_{0,2,\phi}(r_0) = 0$,
$0<r_{0}<R$. If there exists $r^{\ast}\in[0,r_0)$ such that
$T_{0,2,\phi}''(r^{\ast}) \geq 0$ and $T_{0,2,\phi}'(r^{\ast}) = 0$,
then substituting this fact into (\ref{5-7}) yields $
((pT_{0,2,\phi}')'+\mu_{0,2,\phi}pT_{0,2,\phi})(r^{\ast})> 0$, which
contradicts with the first equation in the system (\ref{5-7}).
Hence, we conclude that $T'_{0,2,\phi}$ is negative on the interval
$(0, r_0)$. Since $\phi$ is non-increasing, $p'\geq 0$ can be
obtained, and then from (\ref{5-7}) again, we have
$T''_{0,2,\phi}(r_0) \geq 0$ at $r_0$.

We notice that the function $T_{1,1,j}$ satisfies the following
equation
\begin{eqnarray} \label{5-8}
(pT_{1,1,\phi}')'+(\mu_{1,1,\phi}-(n-1)
S_{\kappa}^{-2})pT_{1,1,\phi}=0.
\end{eqnarray}
Differentiating both sides of the first equation in the system
(\ref{5-7}) results in
\begin{eqnarray} \label{5-9}
(pT_{0,2,\phi}'')'+\left(\mu_{0,2,\phi}+(\frac{p'}{p})'\right)pT'_{0,2,\phi}=0.
\end{eqnarray}
Combining (\ref{5-8})-(\ref{5-9}), and applying Lemma
\ref{lemma5-1}, we can obtain at $r_0$ that
\begin{eqnarray} \label{5-10}
&& p (T_{1,1,\phi} T_{0,2,\phi}'' - T_{1,1,\phi}' T_{0,2,\phi}'
)(r_0)
= \nonumber \\
&& \qquad \qquad \int_0^{r_0} \left[\mu_{1,1,\phi}-\mu_{0,2,\phi} +
\left(\left(-\frac{p'}{p}\right)'-(n-1)S_{\kappa}^{-2}\right)\right]pT_{1,1,\phi}
T_{0,2,\phi}'dt. \qquad
\end{eqnarray}
Since $\phi$ is a convex function, $\phi''\geq0$, and so we have
\begin{eqnarray*}
-\left(\frac{(n-1)C_{\kappa}}{S_{\kappa}}-\phi'\right)'-(n-1)S_{\kappa}^{-2}\geq
0.
\end{eqnarray*}
 Substituting the fact
  \begin{eqnarray*}
p(t)=e^{\int_\epsilon^{t}\left(\frac{(n-1)C_{\kappa}}{S_{\kappa}}-\phi'\right)
ds}
  \end{eqnarray*}
   into the above inequality, one has
\begin{eqnarray*}
\left(-\frac{p'}{p}\right)'-(n-1)S_{\kappa}^{-2}\geq 0.
\end{eqnarray*}
Together with the fact that at $r_0$, $T_{1,1,\phi}>0$,
$T_{1,1,\phi}'\geq 0$, $T_{0,2,\phi}>0$, $T_{0,2,\phi}'\leq 0$ and
$T_{0,2,\phi}''\geq 0$, it follows from (\ref{5-10}) that
$\mu_{1,1,\phi}< \mu_{0,2,\phi}$. That is to say, the first non-zero
Neumann eigenvalue $\mu_{1,\phi}(B_{R}(o))$ of the Witten-Laplacian
on $B_{R}(o)$ should be $\mu_{1,\phi}=\mu_{1,1,\phi}$. Substituting
this fact in (\ref{5-1}) results in
\begin{eqnarray*}
T''+\left[\frac{(n-1)C_{\kappa}}{S_{\kappa}}-\phi'\right]T'+\left(\mu_{1}(B_{R}(o))-v_{1}S_{\kappa}^{-2}\right)T=0,
 \end{eqnarray*}
which is exactly the first equation in the system (\ref{4-1}). This
completes the proof of Theorem \ref{theo4-1}.

\section*{Acknowledgments}
\renewcommand{\thesection}{\arabic{section}}
\renewcommand{\theequation}{\thesection.\arabic{equation}}
\setcounter{equation}{0} \setcounter{maintheorem}{0}

This research was supported in part by the NSF of China (Grant Nos.
11801496 and 11926352), the Fok Ying-Tung Education Foundation
(China), Hubei Key Laboratory of Applied Mathematics (Hubei
University), and Key Laboratory of Intelligent Sensing System and
Security (Hubei University), Ministry of Education. The authors are
grateful to the anonymous referee for careful reading and valuable
comments such that the paper appears as its present version.

\end{document}

\setcounter{page}{1}
\title{Corrigendum to ``Several isoperimetric inequalities of
 Dirichlet and Neumann eigenvalues of the Witten Laplacian"}
\author{ Ruifeng Chen, ~~~  Jing Mao$^{\dag}$}

\date{}
\protect
\footnotetext{\!\!\!\!\!\!\!\!\!\!\!\!{~$^{\dag}$Corresponding
author}\\
{~MSC 2020: 35P15 (primary); 35J10, 35J15 (secondary).} }
\maketitle ~~~\\[-15mm]

\begin{center}
{\footnotesize
Faculty of Mathematics and Statistics, \\
Key Laboratory of Applied Mathematics of Hubei Province, \\
Hubei University,
Wuhan 430062, Hubei, China \\
 E-mails:~gchenruifeng@163.com, jiner120@163.com  }
\end{center}

%\\[5mm]

\vspace{1cm}

\begin{abstract}
Corrigendum to Journal of Spectral Theory 15 (2025) 1241--1277.
\end{abstract}

\vspace{1cm}

\markright{\sl\hfill R. F. Chen, J. Mao\hfill}

The article \cite{CM} proves several spectral isoperimetric
inequalities for the Dirichlet and Neumann eigenvalues of the
Witten-Laplacian. However, there are several chaos in \cite{CM} we
wish to point out. They are:

\begin{itemize}
\item In \cite{CM}, the regularity requirements of the potential function $\phi$ can actually be
relaxed. In fact, in \cite{CM}, for the Dirichlet eigenvalue problem
(1.15) or the Neumann eigenvalue problem (1.20) of the
Witten-Laplacian, we can define a notion ``\textbf{The space of
acceptable functions}" (SAF). For the Dirichlet case, SAF is the
space of all possible functions such that
$W^{1,2}_{0,\phi}(\Omega)\hookrightarrow L^{2}_{\phi}(\Omega)$ is
compact embedding, and the Witten-Laplacian $\Delta_{\phi}$ can be
well-defined in the sense of weak derivatives. For the Neumann case,
SAF is the space of all possible functions such that
$W^{1,2}_{\phi}(\Omega)\hookrightarrow L^{2}_{\phi}(\Omega)$ is
compact embedding, and the Witten-Laplacian $\Delta_{\phi}$ can be
well-defined in the sense of weak derivatives. By the way,
$W^{1,2}_{0,\phi}(\Omega)$, $W^{1,2}_{\phi}(\Omega)$ and
$L^{2}_{\phi}(\Omega)$ are weighted functional spaces well-defined
in \cite[Section 1]{CM}. The compact embedding in the Dirichlet case
(or in the Neumann case) is used to assure the existence of the
discrete spectrum for the problem (1.15) (or for the problem
(1.20)). There are many kinds of functions satisfying the SAF
condition. For instance, $\phi\in C^{\infty}(\overline{\Omega})$,
$\phi\in W^{1,\infty}(\Omega)$, $\phi\in W^{2,\infty}(\Omega)$, and
so on.

\item Once the notion SAF is defined, the condition that ``$\phi\in
C^{\infty}(M^n)$ is real-valued on a given complete $n$-manifold
$M^n$" mentioned in p. 1247 of \cite{CM} can be weakened to be that
``\emph{$\phi$ is a real-valued function defined on $\Omega$ and
satisfying the SAF condition}". Besides, in \cite[Theorem 1.1]{CM},
we have assumed that the potential function $\phi$ satisfies
Property 1 (i.e. $\phi$ is radial) and is concave. The concavity of
$\phi$ is required since in the proof of \cite[Theorem 1.1]{CM}, we
have used the conclusion of the Log-convex density conjecture proved
by G. R. Chambers (see [JEMS, 21 (2019) 2301--2332]). That is,
$\log(e^{-\phi(t)})''\geq0$, which is equivalent to
$\phi''(t)\leq0$, i.e. $\phi$ is concave (in the radial direction).
Moreover, the usage of Chambers' conclusion implies that $\phi$ is
also smooth in the radial direction. It is easy to check that our
assumptions made for the potential function $\phi$ let $\phi$
naturally belongs to SAF for the Dirichlet case. Similarly, for the
Neumann case, it is not hard to see that our assumptions made for
$\phi$ implicitly include that $\phi$ is at least $C^2$ in the
radial direction, and it is also easy to check that our assumptions
made for $\phi$ let $\phi$ naturally belongs to SAF for the Neumann
case.

\item In \cite[Theorem 1.1]{CM}, the characterization of the equality case should be:
the equality holds if (up to measure zero) $\Omega$ is the ball
$B_R(o)$. This is because a ball, with the same weighted volume and
entirely contained within $B_{\mathcal{R}(h)}$, can also achieve the
lower bound of $\lambda_{1,\phi}$. If we require the potential
function $\phi$ to be strictly concave, then the equality holds if
and only if (up to measure zero) $\Omega$ is the ball $B_R(o)$. In
addition, in the derivation from (2.11) to (2.12), since the
constant mean curvature condition is a boundary condition and cannot
be applied to the boundaries of each level set, the result of
Theorem 1.3 would just hold for spherical caps.

\item Finally, we wish to give two non-trivial examples for the potential
function $\phi$ which satisfy all the assumptions in Theorems 1.11
and 1.12. They are $\phi(t)=e^{-t}$ and $\phi(t)=\frac{1}{t+1}$,
with $t\in[0,\infty)$. We believe that readers would construct more
non-trivial examples by using AI tools in this era of artificial
intelligence. Additionally, in order to ensure that the weighted
volume remains invariant and that $h(t)\frac{x_i}{t}$ can serve as a
test function for $\mu_{1,\phi}(\Omega)$ (see \cite[p. 1269]{CM}),
we should further add the assumption that ``\emph{the domain
$\Omega$ is centrally symmetric with respect to the point $o$}" in
Theorems 1.11 and 1.12.

\end{itemize}

% \section*{Acknowledgements}

%We would like to express our sincere gratitude to the anonymous
%referee for his or her valuable suggestions such that this erratum
%appears as its present version.